\documentclass[fleqn,reqno,11pt,a4paper,final]{amsart}

\usepackage[a4paper,left=30mm,right=30mm,top=30mm,bottom=30mm,marginpar=23mm]{geometry} 
\usepackage{amsmath}
\usepackage{amssymb}
\usepackage{amsthm}
\usepackage{amscd}
\usepackage[ansinew]{inputenc}
\usepackage{cite}
\usepackage{bbm}
\usepackage{color}
\usepackage[english=american]{csquotes}
\usepackage[final]{graphicx}
\usepackage{hyperref}
\usepackage{calc}
\usepackage{mathptmx}

\graphicspath{{../Pictures/}}

\numberwithin{equation}{section}

\newtheoremstyle{thmlemcorr}{10pt}{10pt}{\itshape}{}{\bfseries}{.}{10pt}{{\thmname{#1}\thmnumber{ #2}\thmnote{ (#3)}}}
\newtheoremstyle{thmlemcorr*}{10pt}{10pt}{\itshape}{}{\bfseries}{.}\newline{{\thmname{#1}\thmnumber{ #2}\thmnote{ (#3)}}}
\newtheoremstyle{remexample}{10pt}{10pt}{}{}{\bfseries}{.}{10pt}{{\thmname{#1}\thmnumber{ #2}\thmnote{ (#3)}}}
\newtheoremstyle{ass}{10pt}{10pt}{}{}{\bfseries}{.}{10pt}{{\thmname{#1}\thmnumber{ A#2}\thmnote{ (#3)}}}

\theoremstyle{thmlemcorr}
\newtheorem{theorem}{Theorem}
\numberwithin{theorem}{section}

\newtheorem{corollary}[theorem]{Corollary}
\newtheorem{proposition}[theorem]{Proposition}

\newtheorem{definition}[theorem]{Definition}

\theoremstyle{thmlemcorr*}
\newtheorem{theorem*}{Theorem}
\newtheorem{lemma*}[theorem]{Lemma}
\newtheorem{corollary*}[theorem]{Corollary}
\newtheorem{proposition*}[theorem]{Proposition}
\newtheorem{problem*}[theorem]{Problem}
\newtheorem{conjecture*}[theorem]{Conjecture}
\newtheorem{definition*}[theorem]{Definition}

\theoremstyle{remexample}
\newtheorem{remark}[theorem]{Remark}

\theoremstyle{ass}

\newcommand{\Crm}{\mathrm{C}}

\newcommand{\Lrm}{\mathrm{L}}

\newcommand{\Wrm}{\mathrm{W}}

\newcommand{\Acal}{\mathcal{A}}

\newcommand{\Fcal}{\mathcal{F}}

\newcommand{\Lcal}{\mathcal{L}}

\newcommand{\Rcal}{\mathcal{R}}

\newcommand{\Mbf}{\mathbf{M}}

\renewcommand{\Bbb}{\mathbb{B}}

\newcommand{\Nbb}{\mathbb{N}}

\newcommand{\Rbb}{\mathbb{R}}

\DeclareMathOperator{\supp}{supp}

\newcommand{\ee}{\mathrm{e}}

\newcommand{\setn}[2]{\{\, #1 \ \ \textup{\textbf{:}}\ \ #2 \,\}}
\newcommand{\setb}[2]{\bigl\{\, #1 \ \ \textup{\textbf{:}}\ \ #2 \,\bigr\}}

\newcommand{\norm}[1]{\|#1\|}

\newcommand{\abs}[1]{|#1|}

\newcommand{\absBB}[1]{\biggl|#1\biggr|}

\newcommand{\dd}{\;\mathrm{d}}

\newcommand{\N}{\mathbb{N}}
\newcommand{\R}{\mathbb{R}}

\newcommand{\ONE}{\mathbbm{1}}

\newcommand{\toweak}{\rightharpoonup}

\newcommand{\toup}{\uparrow}

\newcommand{\sbullet}{\begin{picture}(1,1)(-0.5,-2.5)\circle*{2}\end{picture}}
\newcommand{\frarg}{\,\sbullet\,}

\newcommand{\toY}{\overset{\mathrm{Y}}{\to}}

\newcommand{\lrangle}[1]{\langle #1 \rangle}

\newcommand{\term}[1]{\textbf{#1}}

\newcommand{\proofstep}[1]{\textit{#1}}

\def\Xint#1{\mathchoice 
{\XXint\displaystyle\textstyle{#1}}%
{\XXint\textstyle\scriptstyle{#1}}%
{\XXint\scriptstyle\scriptscriptstyle{#1}}%
{\XXint\scriptscriptstyle\scriptscriptstyle{#1}}%
\!\int} 
\def\XXint#1#2#3{{\setbox0=\hbox{$#1{#2#3}{\int}$} 
\vcenter{\hbox{$#2#3$}}\kern-.5\wd0}} 
\def\dashint{\,\Xint-}

\newcommand{\restrict}{\begin{picture}(10,8)\put(2,0){\line(0,1){7}}\put(1.8,0){\line(1,0){7}}\end{picture}}

\renewcommand{\epsilon}{\varepsilon}
\renewcommand{\phi}{\varphi}

\begin{document}


\title[Differential inclusions and Young measures involving prescribed Jacobians]{Differential inclusions and Young measures involving prescribed Jacobians}

\author{Konstantinos Koumatos}
\address{\textit{Konstantinos Koumatos:} Mathematical Institute, University of Oxford, Andrew Wiles Building, Radcliffe Observatory Quarter, Woodstock Road, Oxford OX2 6GG, United Kingdom.}
\email{koumatos@maths.ox.ac.uk}

\author{Filip Rindler}
\address{\textit{Filip Rindler:} Mathematics Institute, University of Warwick, Coventry CV4 7AL, United Kingdom, United Kingdom.}
\email{F.Rindler@warwick.ac.uk}

\author{Emil Wiedemann}
\address{\textit{Emil Wiedemann:} Hausdorff Center for Mathematics and Mathematical Institute, Universit\"{a}t Bonn, Endenicher Allee 60, 53115 Bonn, Germany.}
\email{emil.wiedemann@hcm.uni-bonn.de}


\hypersetup{
  pdfauthor = {Konstantinos Koumatos (University of Oxford) and Filip Rindler (University of Warwick) and Emil Wiedemann (University of British Columbia and PIMS)},
  pdftitle = {Differential inclusions and Young measures involving prescribed Jacobians},
  pdfsubject = {MSC (2010): 49J45 (primary); 28B05, 46G10},
  pdfkeywords = {Gradient Young measure, convex integration, positive Jacobian, incompressibility, laminates}
}


\maketitle


\begin{abstract}
This work presents a general principle, in the spirit of convex integration, leading to a method for the characterization of Young measures generated by gradients of maps in $\Wrm^{1,p}$ with $p$ less than the space dimension, whose Jacobian determinant is subjected to a range of constraints. Two special cases are particularly important in the theories of elasticity and fluid dynamics: (a) the generating gradients have positive Jacobians that are uniformly bounded away from zero and (b) the underlying deformations are incompressible, corresponding to their Jacobian determinants being constantly one. This characterization result, along with its various corollaries, underlines the flexibility of the Jacobian determinant in subcritical Sobolev spaces and gives a more systematic and general perspective on previously known pathologies of the pointwise Jacobian. Finally, we show that, for $p$ less than the dimension, $\Wrm^{1,p}$-quasiconvexity and $\Wrm^{1,p}$-orientation-preserving quasiconvexity are both unsuitable convexity conditions for nonlinear elasticity where the energy is assumed to blow up as the Jacobian approaches zero.
\vspace{4pt}



\vspace{4pt}

\noindent\textsc{Date:} \today{} 
\end{abstract}

%

\section{Introduction}

In this work we continue the investigation started in~\cite{KoRiWi13OPYM} into the positive Jacobian constraint in the Calculus of Variations. There, using a convex integration-type argument, we characterized all Young measures generated by sequences in $\Wrm^{1,p}(\Omega;\R^d)$, where $\Omega \subset \R^d$ is a bounded open set and $p<d$, with the property that every element of the sequence has positive Jacobian almost everywhere. 

Here we extend this characterization to more restrictive pointwise constraints on the Jacobian determinant, e.g.\ the condition that it be bounded below by a positive constant or even be equal to a given positive constant almost everywhere. These requirements are very natural in elasticity theory, where they correspond to limited compressibility or incompressibility of an elastic solid. 

On a more theoretical level, our characterization and its various corollaries display the vast flexibility of the pointwise Jacobian determinant in Sobolev spaces $\Wrm^{1,p}(\Omega;\R^d)$ below the critical exponent $p=d$. While it is well-known \cite{Ball77CCET, BalMur84WQVP,Sver88RPDF,Mull90DETD,Mull93SSDD,Hen11SobHomJ0} that the Jacobian loses many of its usual geometric properties for $p<d$, thus leading to the failure of weak continuity or of the change-of-variables formula (i.e., in terms of elasticity theory, to cavitation), one of our aims in this work is to systematize and generalize these observations within a convex integration framework.

We refer to Sections~\ref{statement} and~\ref{sc:applications} below for a precise formulation of our results. Before that, however, we wish to give an informal discussion of our findings, highlighting various different aspects. 

\subsection{Kinderlehrer--Pedregal theory}
It is a recurrent theme in the Calculus of Variations to obtain characterization results for Young measures generated by sequences of maps with specific properties. The prototypical result is that of Kinderlehrer--Pedregal~\cite{KinPed91CYMG,KinPed94GYMG}, which applies to sequences of gradients. Various generalizations have been studied, e.g.\ to so-called $\mathcal{A}$-free sequences~\cite{FonMul99AQLS} or generalized Young measures involving concentrations~\cite{FoMuPe98ACOE,KriRin10CGGY,Rind14LPCY}. An additional difficulty is posed by requiring that the generating sequence satisfies not only a linear differential constraint (like the gradient constraint), but also a nonlinear and nonconvex pointwise constraint. Such a problem was treated in~\cite{SzeWie12YMGI}, where the constraint was related to the incompressible Euler equations, and in~\cite{KoRiWi13OPYM}, where the Jacobian determinant was required to be positive almost everywhere. This article presents a significant extension of the latter result, see Theorem~\ref{thm:main_intro} below. Note, however, that~\cite{KoRiWi13OPYM} is not strictly contained in the present work as the side constraint is \emph{open} in \emph{loc.\ cit.} and \emph{closed} here.

\subsection{First-order PDEs}  
A corollary of our characterization (Theorem~\ref{damo}) is an existence statement for Dirichlet problems of the form 
\begin{equation}\label{damointro} 
\left\{
\begin{aligned}
\det\nabla v(x)&=J(x),\\
u|_{\partial\Omega}&=g.
\end{aligned}
\right.
\end{equation}
This problem was first stated in this form by B. Dacorogna and J. Moser~\cite{DacMos90PDEJ}, motivated by earlier work of Moser~\cite{Mose65OVEM} on diffeomorphisms between volume forms on manifolds. They answered the existence question positively provided $g=\operatorname{id}$ and $J$ is positive, lies in $C^{k,\alpha}$, and satisfies a compatibility condition. Their solution $v$ then is a $C^{k+1,\alpha}$-diffeomorphism. When the positivity assumption on $J$ is dropped or different boundary conditions are considered, similar results are available~\cite{CuDaKn09OENS, Kneu12OEPB}, but then $v$ may no longer be chosen as a diffeomorphism. For a similar result and a discussion of this problem in Sobolev spaces we refer the reader to \cite{Ye94PJSP}.

Here, in Section~\ref{ssc:DM}, we establish the following result: If $g\in \Wrm^{1-1/p,p}(\partial\Omega)$ for some $1 < p<d$ and $J\in\Lrm^{p/d}(\Omega)$ is measurable, then there exists a solution $v\in \Wrm^{1,p}(\Omega)$ of~\eqref{damointro}. The fact that our result requires no compatibility condition on $J$ and $g$ underscores the pathological behaviour of the Jacobian for $p<d$ and the loss of its classical geometric properties.

\subsection{The distributional determinant}
The properties of the pointwise determinant of matrix-valued maps in $\Lrm^p$, $p<d$, led to the definition of the \term{distributional determinant}~\cite{Ball77CCET}, which may no longer be defined as a function, but only as a distribution. In~\cite{Mull93SSDD} examples were constructed of maps for which the difference of the distributional and the pointwise determinant is supported on sets of arbitrary Hausdorff dimension $\alpha\in(0,d)$. We also exhibit in the present paper, by completely different methods, examples of maps whose distributional and pointwise determinants differ (\emph{any} solution of~\eqref{damointro} with $g=\operatorname{id}$ and $\int_{\Omega}J(x) \dd x \neq |\Omega|$ will have this property).

Of course, our results do not answer the intriguing problem~\cite{Mull93SSDD} under what conditions~\eqref{damointro} can be solved in $\Wrm^{1,p}$ if one replaces the pointwise determinant by the distributional one. In fact it would be interesting to know what can be said about the distributional Jacobian of the maps that we construct.

\subsection{Cavitation}
A related phenomemon in elasticity theory is \term{cavitation}~\cite{Ball82DESC,Sver88RPDF,MuSp95EX,SivSp00EX,HeMo-Co10CAV&FRAC}, which refers to the formation of holes in an elastic solid. Consider the problem~\eqref{damointro} with $\Omega=B_1(0)$ (the unit ball in $\mathbb{R}^d$), $J\equiv1$, and $g(x)=2x$. If an elastic solid is to be deformed according to this data, the deformation necessarily has to be discontinuous, thus exhibiting cavitation. Since our convex integration construction is in a sense local and does not distinguish particular points in the domain, the discontinuous solutions in $W^{1,p}$ produced in this paper include a kind of ``diffuse cavitation''.

A further consequence of these observations in conjunction with~\cite{Sver88RPDF} is that for the maps we construct, the cofactor matrix $\operatorname{cof}\,\nabla v$, which is easily seen to be in $\Lrm^{p/(d-1)}$, cannot be expected to lie in general in $\Lrm^q$ for any $q\geq p/(p-1)$. 

\subsection{Relaxation}
We prove a relaxation theorem (Corollary~\ref{thm:relaxation} below) under the constraint that the gradients of admissible maps have determinants greater or equal to $r>0$, or determinants precisely equal to $r$, almost everywhere. This follows immediately from our results. No relaxation results under these constraints seem to exist in the literature. We note~\cite{AnHMan09RTNE} where a relaxation theorem is proved for $p\in(1,\infty)$ under the assumption that the integrand $f$ satisfies $f(A)\to\infty$ as $\det\, A\to 0^+$, nevertheless without accounting for the requirement that $f(A)=\infty$ if $\det A\leq 0$ which is natural in elasticity. A very interesting relaxation result was also recently proved in \cite{ContiDolz14} for functionals relevant in elasticity theory and $p\geq d$. Of course it would be very interesting to find similar relaxation results with the pointwise Jacobian replaced by the distributional one.

\subsection{Weak continuity of the determinant}
It is well-known that if $u_j \toweak u$ in $\Wrm^{1,p}$ with $p\geq d$, then $\det\nabla u_j\toweak \det\nabla u$ in the sense of distributions, whereas this weak continuity property may fail for $p<d$ (see e.g.~\cite{BalMur84WQVP,FoLeMa05WCLS} and the references therein). This is again related to the discrepancy between the pointwise and the distributional determinant. In fact, for $p<d$, it is shown in \cite[Ex.~3, p.~284]{GMSvol1} that the map $u(x)=x$ can be approximated weakly in $\Wrm^{1,p}$ by a sequence $(u_j)$ such that $\det\nabla u_j =0$ a.e., making the determinant weakly discontinuous in $\Wrm^{1,p}$. The same result can be extended to any smooth function $u$ (see \cite{DePhil12WJ}) and by density to $u\in \Wrm^{1,p}$, so that the determinant is weakly discontinuous everywhere in $\Wrm^{1,p}$.

In Corollary~\ref{cor:approximation} we strongly exhibit this everywhere discontinuity by showing that any $u\in\Wrm^{1,p}$ ($p<d$) can be approximated weakly in $\Wrm^{1,p}$ by a sequence of maps with Jacobian even prescribed almost everywhere.

\subsection{Lower Semicontinuity}
As a final application, in Section~\ref{sec:lsc} we make the perhaps surprising observation that, for $p<d$, neither the class of $\Wrm^{1,p}$-quasiconvex stored-energy functions, cf.~\cite{BalMur84WQVP}, nor the seemingly larger class of $\Wrm^{1,p}$-orientation-preserving quasiconvex functions are suitable for the minimization problems of nonlinear elastostatics under realistic growth assumptions. We accomplish this by showing that an integrand cannot be $\Wrm^{1,p}$-(orientation-preserving) quasiconvex and satisfy natural growth conditions at the same time. In particular, this essentially rules out that $\Wrm^{1,p}$-orientation-preserving quasiconvex functions can satisfy the condition
\[
  f(A) \to \infty  \quad\text{as}\quad  \det\, A \to 0^+\quad\mbox{ and }\quad f(A) = \infty  \quad\text{if}\quad  \det\, A \leq 0,
\]
which one imposes on realistic integrands in nonlinear elasticity, see~\cite{Ball02SOPE}. In this context, we remark that the energies are formulated in terms of the \emph{pointwise} Jacobian.

\subsection{Convex integration}     
Finally, we give some remarks on the method of proof of our results, which can be viewed as an instance of \term{convex integration}. In this general technique, one uses an iteration scheme which starts, in our case, from any map $u\in\Wrm^{1,p}$ and approaches the determinant constraint by adding suitable oscillatory perturbations at each step, the frequencies increasing rapidly from step to step. The crucial observation (Proposition~\ref{prop:geometry}) is that the \term{$p$-quasiconvex hull} (cf.~\cite{KoRiWi13OPYM}) of the set of matrices with given determinant is sufficiently large such as to provide for enough suitable perturbations. 

Convex integration has been used in a variety of situations in topology, differential geometry, nonlinear PDE, and the Calculus of Variations~\cite{Nash54C1IE,Grom86PDR,DacMar97GETH,EliMis02IH,Kirc03RGM,MulSve03CILM,AsFaSz08CILP,DeLSze12HEFD}. A common feature of these very different problems is that there exists a ``threshold regularity'' above which the situation is ``rigid'', whereas below the threshold the problem displays ``flexible'' behavior. For instance, the only $\Crm^2$ isometric embedding of $\mathbb{S}^2$ into $\mathbb{R}^3$ is the canonical embedding, whereas J.~Nash~\cite{Nash54C1IE} constructed infinitely many $\Crm^1$ embeddings with unexpected behavior. The loss of rigidity in this example is due to the lack of a well-defined curvature of the embedded submanifold. Another, more recent, example is given by the incompressible Euler equations~\cite{DeLSze12HEFD}. There, the kinetic energy is conserved (``rigidity'') for sufficiently regular solutions, but can become subject to dissipation (``flexibility'') for less regular solutions.

Similarly, the main thrust of this work entails that the Jacobian determinant is ``rigid'' in $W^{1,p}$ for $p\geq d$, yet becomes ``flexible'' and loses many of its classical properties for $p<d$, as showcased in the course of our previous discussion.

On a more technical level, our method allows one to distinguish via convex integration between different levels of Sobolev regularity (the only other results of this kind, as far as we are aware, are found in~\cite{AsFaSz08CILP} and the work of Yan~\cite{Yan96RemarksStability,Yan01LinearBVP,Yan03Baire}). Moreover, our convergence argument via Young measure generation (proof of Proposition~\ref{prop:convexint}) is new and may be helpful to facilitate future convex integration-type arguments.    

\subsection{Plan of the paper}
The plan of this paper is as follows: First, in Section~\ref{sc:setup}, we give a brief introduction to Young measures and introduce terminology. Section~\ref{statement} gives a precise formulation of the main characterization result. In Section~\ref{sc:convex_int}, we provide all the necessary definitions and present our convex integration principle (Proposition~\ref{prop:convexint}) leading to a general method (Theorem~\ref{thm:main_abstract}) for the characterization of Young measures generated by gradients that satisfy a differential inclusion of the form $\nabla u(x) \in S_{R(x,\frarg)}$ a.e., where $S_{R(x,\frarg)}$ is the zero-sublevel set of a Carath\'{e}odort function $R$. For this, we require a \enquote{tightness condition} on the $p$-quasiconvex hull of $S_{R(x,\frarg)}$ (see Definition~\ref{def:p-full}).

In Section~\ref{sc:geometry}, we restrict attention to constraint functions of the form $R(x,A)=\max\{J_1(x) - \det\, A, \det\, A - J_2(x), 0\}$ with corresponding sublevel set $S_{R(x,\frarg)} = \setb{A\in\R^{d\times d}}{J_1(x)\leq \det\, A \leq J_2(x)}$. We prove that the above sets satisfy the hypotheses of Theorem~\ref{thm:main_abstract} and the characterization of the corresponding gradient Young measures follows. For the convenience of the reader, the result is first proved in the physically most relevant dimension $d=3$. The case $d=2$ is significantly simpler and the proof is omitted, whereas the case $d>3$ is, at least notationally, more involved and is presented separately in Section~\ref{sc:arbitrary_dim}. Sections~\ref{sc:applications} and~\ref{sec:lsc} are devoted to the applications mentioned above.

\subsection*{Acknowledgments}

The authors wish to thank John Ball, Sergio Conti, Georg Dolzmann and Jan Kristensen for discussions related to the present paper. KK was supported by the European Research Council grant agreement ${\rm n^o}$ 291053. FR and EW were partly supported by a Royal Society International Exchange Grant IE131532.

\section{Setup}\label{sc:setup}

On the space $\R^{d \times d}$ of $(d \times d)$-matrices $M = (M^i_j)$ ($i,j = 1,\ldots,d$) we use the Frobenius norm
\[
  \abs{M} = \abs{M}_F := \left[ \sum_{i,j=1}^d (M^i_j)^2 \right]^{1/2}
  = \left[ \sum_{k=1}^d \sigma_k^2 \right]^{1/2},
\]
where $\sigma_k$, $k=1,\ldots,d$ are the singular values of $M$.

Let $1 \leq p < \infty$. A \term{$p$-Young measure} is a parametrized family $\nu = (\nu_x)_{x \in \Omega} \subset \Mbf^1(\R^N)$ of probability measures on $\R^N$, where $\Mbf^1(\R^N)$ denotes the space of probability measures, such that:
\begin{enumerate}
\item[(1)] The family $(\nu_x)$ is \term{weakly* measurable}, that is, for every Borel set $B \subset \R^N$ the map $x \mapsto \nu_x(B)$ is ($\Lcal^d \restrict \Omega$)-measurable.
\item[(2)] The map $x \mapsto \int \abs{A}^p \dd \nu_x$ lies in $\Lrm^1(\Omega)$.
\end{enumerate}
Many properties of Young measures are collected in~\cite{Pedr97PMVP}, we recall only some of them: The \term{barycenter} of a $p$-Young measure $\nu$ is
\[
  [\nu](x) := \int A \dd \nu_x(A),  \qquad x \in \Omega,
\]
and $[\nu] \in \Lrm^p(\Omega;\R^N)$. A Young measure $\nu$ is \term{homogeneous} if $x \mapsto \nu_x$ is an almost everywhere constant map, i.e.\ $\nu_x = \nu \in \Mbf^1(\R^N)$ for a.e.\ $x \in \Omega$.

We say that a (necessarily norm-bounded) sequence $(u_j) \subset \Lrm^p(\Omega;\R^N)$ \term{generates} the Young measure $\nu$ if 
\[
  \int_\Omega f(x,u_j(x)) \dd x  \;\;\to\;\;
  \int_\Omega \int f(x,A) \dd \nu_x(A) \dd x 
  \]
for all Carath\'{e}odory functions $f:\Omega\times\R^N\to \R$ (that is, $f$ is measurable in the first and continuous in the second argument) such that $(f(\frarg,u_j))$ is equiintegrable. We express generation in symbols as $u_j\toY \nu$.

It can be shown that if $(u_j)$ and $(v_j)$ are $\Lrm^p(\Omega)$-bounded sequences with $\norm{u_j-v_j}_p \to 0$ as $j\to\infty$ and $(u_j)$ generates the Young measure $\nu$, then also $(v_j)$ generates $\nu$. It can further be proved that all $p$-Young measures are generated by some sequence of uniformly $\Lrm^p(\Omega;\R^N)$-bounded functions.

\section{Statement of the main result}\label{statement}
We consider in this article \term{differential inclusions} of the form
\begin{equation}
\label{eq:constraint_intro}
  \nabla u(x) \in S_{R(x,\frarg)} \quad\text{a.e.,}\qquad
  \nabla u \in \Lrm^p(\Omega,\R^{d\times d}),
\end{equation}
where $S_{R(x,\frarg)}$ is the zero-sublevel set of $R(x,\frarg)$ for a Carath\'{e}odory \term{constraint function} $R \colon\Omega\times\R^{d \times d} \to \R$, i.e.
\[
S_{R(x,\frarg)} := \setb{A\in\R^{d\times d}}{R(x,A)\leq 0}.
\]

This principle generalizes some of the methods presented in~\cite{KoRiWi13OPYM} to arbitrary constraints of the form \eqref{eq:constraint_intro} satisfying certain properties (see Definition~\ref{def:p-full}).
As an application of the general principle, we provide a characterization of Young measures generated by gradients bounded in $\Lrm^p(\Omega,\R^{d\times d})$, $1<p<d$, and satisfying a constraint of the form
\begin{equation} \label{eq:det_r}
  J_1(x) \leq \det \nabla u_j(x) \leq J_2(x) \qquad\text{for all $j$ and a.e.\ $x\in\Omega$,}
\end{equation}
where
\[
  J_1 \colon \Omega \to [-\infty,+\infty), \qquad 
  J_2 \colon \Omega \to (-\infty,+\infty]
\]
are given functions such that
\[
  J_1(x)\leq J_2(x) \qquad\text{for a.e.~$x\in\Omega$.}
\]

This characterization gives rise to a number of special cases which are discussed after the statement of the main result:

\begin{theorem}\label{thm:main_intro}
Let $1 < p < d$. Suppose that $\Omega \subset \R^d$ is open and bounded, $|\partial\Omega|=0$, and let $\nu = (\nu_x)_{x\in\Omega} \subset \Mbf^1(\R^{d \times d})$ be a $p$-Young measure. Moreover let $J_1 \colon \Omega \to [-\infty,+\infty)$, $J_2 \colon \Omega \to (-\infty,+\infty]$ be measurable and such that $J_1(x)\leq J_2(x)$ for a.e.~$x\in\Omega$. Also, assume that for $i=1,2$,
\begin{equation*}
\int_{\Omega}J_1^+(x)^{p/d} \dd x<\infty
\qquad\text{and}\qquad
\int_{\Omega}J_2^-(x)^{p/d} \dd x<\infty,
\end{equation*} 
where $J_i^{\pm}$ denotes the positive or negative part of $J_i$, respectively. Then the following statements are equivalent:
\begin{itemize}
  \item[(i)] There exists a sequence of gradients $(\nabla u_j) \subset \Lrm^p(\Omega;\R^{d \times d})$ that generates $\nu$, such that 
\[
  \qquad J_1(x) \leq \det \nabla u_j(x) \leq J_2(x) \quad\text{for all $j\in\N$ and a.e.~$x\in\Omega$. }
\]
  \item[(ii)] The conditions (I)-(IV) hold:
\begin{itemize}
\item[(I)] $\displaystyle\int_{\Omega}\int\abs{A}^p \dd \nu_x(A)<\infty$;
\item[(II)] the barycenter $[\nu](x) := \int A \dd \nu_x(A)$ is a gradient, i.e.\ there exists $\nabla u \in \Lrm^p(\Omega;\R^{d \times d})$ with $[\nu] = \nabla u$ a.e.;
\item[(III)] for every quasiconvex function $h \colon \R^{d \times d} \to \R$ with $\abs{h(A)}\leq c(1+\abs{A}^p)$, the Jensen-type inequality
\[
  \qquad\qquad h(\nabla u(x)) \leq  \int h(A) \dd \nu_x(A)  \qquad\text{holds for a.e.\ $x \in \Omega$;}
\]
\item[(IV)] $\supp{\nu_x} \subset \setb{ A \in \R^{d \times d} }{ J_1(x) \leq \det\, A \leq J_2(x) }$ for a.e.\ $x \in \Omega$.
\end{itemize}
\end{itemize}
Furthermore, in this case the sequence $(u_j)$ can be chosen such that $(\nabla u_j)$ is $p$-equiintegrable\footnote{We say that a sequence $(\nabla u_j)$ is $p$-equiintegrable if the sequence $(|\nabla u_j|^p)$ is equiintegrable.} and $u_j - u\in\Wrm^{1,p}_0(\Omega,\mathbb{R}^d)$, where $u\in\Wrm^{1,p}(\Omega,\R^d)$ is the deformation underlying $\nu$ (i.e.\ the function whose gradient is the barycenter of $\nu$).
\end{theorem}

Recall that a locally bounded Borel function $h \colon \R^{d \times d} \to \R$ is called \term{quasiconvex} if
\begin{equation} \label{eq:quasiconvex}
  h(A_0) \leq \dashint_{\Bbb^d} h(\nabla v(x)) \dd x
\end{equation}
for all $A_0 \in \R^{d \times d}$ and all $v \in \Crm^\infty(\Bbb^d;\R^d)$ with $v(x) = A_0x$ on $\partial \Bbb^d$, see~\cite{Daco08DMCV} for more on this fundamental class of functions.

The conditions~(I)-(III) are the well-known criteria of Kinderlehrer--Pedregal~\cite{KinPed91CYMG, KinPed94GYMG} characterizing gradient $p$-Young measures. Observe also that the conditions on $J_i$ only concern the sets where the functions are finite and hence the constraints are active. For example, if $J_1 \equiv -\infty$, then the lower bound is inactive and the condition on $J_1$ is trivially true.

As an important special case, for $J_1(x) = J_2(x) = J(x)$ a.e.~in $\Omega$, measurable and in $\Lrm^{p/d}$, we obtain a characterization under a constraint of the Dacorogna--Moser~\cite{DacMos90PDEJ} form
\[
\left\{
\begin{aligned}
  &\det \nabla u_j(x) = J(x) \qquad\text{for all $j$ and a.e.\ $x\in\Omega$} \\
  &\text{$J \colon \Omega \to \R$ a given function.}
\end{aligned}
\right.
\]
Moreover, the generating sequence also satisfies $u_j - u \in \Wrm^{1,p}_0(\Omega,\R^d)$ where $u$ is the deformation underlying $\nu$.

In the cases relevant to elasticity, we choose $J_1(x) = r > 0$ a.e.~and $J_2 \equiv +\infty$, corresponding to a uniform positivity constraint on the Jacobian, i.e.
\[
  \det\nabla u_j(x) \geq r > 0  \qquad\text{for all $j$ and a.e.\ $x\in\Omega$.}
\]
In this context we note that requiring the Jacobian to be not only positive but uniformly positive, is often the appropriate condition when considering stored-energy functions $f$ under realistic growth conditions, i.e.
\[
  f(A) \to \infty  \quad\text{as}\quad  \det\, A \to 0^+\quad\mbox{ and }\quad f(A) = \infty  \quad\text{if}\quad  \det\, A \leq 0,
\]
see e.g.~\cite{Ball02SOPE}.

We stress that $p < d$ is necessary for our results to hold. This restriction, however, includes for instance the prototypical case of quadratic growth in three dimensions. This constraint comes as a consequence of the $d$-growth of the determinant, cf.~the discussion in~\cite{KoRiWi13OPYM} and also Remark~\ref{rk:remark_qc_hull} below.

Furthermore, choosing $J_1(x) = J_2(x) = 1$ a.e.~our result also pertains to Young measures generated by gradients of \enquote{incompressible} maps, i.e.
\[
\det \nabla u_j(x) = 1  \qquad\text{for all $j$ and a.e.\ $x\in\Omega$.}
\]
This constraint is particularly relevant in the study of solids and fluids. We remark again, however, that the terminology \enquote{incompressibility} should only be interpreted as a pointwise Jacobian constraint and not as a geometric condition.

The proofs of our results are based on two main pillars: On one hand, an explicit construction of laminates in matrix space allows us to build special homogeneous gradient Young measures expressing an arbitrary matrix as a hierarchy of oscillations along rank-one lines, see Section~\ref{sc:geometry}. On the other hand, the abstract convex integration principle mentioned above then enables us to construct generating sequences consisting of gradients and such that the aforementioned differential inclusions are satisfied \emph{exactly} (it is of course easy to satisfy them only approximately, but the real challenge is to make them satisfied exactly; cf. e.g.\ Chapter~5 of~\cite{Mull99VMMP}). This is contained in Section~\ref{sc:convex_int}.

\section{A general convex integration principle}
\label{sc:convex_int}

In order to state our convex integration principle and main result, we need two definitions:

\begin{definition}
For $1\leq p,q <\infty$, we denote by $\Rcal^{p,q}(\Omega;\R^{d \times d})$ the class of all Carath\'{e}odory functions $R:\Omega\times\R^{d\times d}\to \R$ for which there exists a measurable function $\kappa \colon \Omega\to [0,\infty)$ and a constant $C>0$ (independent of $x$) such that
\[
\int_{\Omega}\kappa(x)^{p/q} \dd x<\infty
\qquad\text{and}\qquad
|R(x,A)|\leq \kappa(x)+C|A|^q.
\]
\end{definition}

\begin{definition} \label{def:p-full}
Suppose $R\in\Rcal^{p,q}(\Omega,\R^{d\times d})$, where $\Omega \subset \R^d$, $1\leq p,q <\infty$, and let
\[
  S_{R(x,\frarg)}:=\setb{A\in\R^{d\times d}}{R(x,A)\leq 0}.
\]
For fixed $x \in \Omega$, we say that a set $D\supseteq S_{R(x,\frarg)}$ is contained \textbf{tightly} in the $p$-quasiconvex hull of $S_{R(x,\frarg)}$ if there exists a constant $C>0$ such that for every $M\in D$ there is a homogeneous gradient $p$-Young measure $\nu$ satisfying the following properties:
\begin{itemize}
\item[(i)] $[\nu] = M$;
\item[(ii)] $\supp \nu \subset S_{R(x,\frarg)}$;
\item[(iii)] $\displaystyle\int |A - M|^p\dd \nu(A) \leq C\max\left\{R(x,M),0\right\}^{p/q}$.
\end{itemize}
We say that a set $D\supseteq \bigcup_{x \in \Omega} S_{R(x,\frarg)}$ is contained tightly in the $p$-quasiconvex hull of $(S_{R(x,\frarg)})_{x\in\Omega}$ \textbf{uniformly} in $x$ if there exists a constant $C>0$ (independent of $x$) such that for every map $M \colon \Omega\to D$ with $M=\nabla u$ for some $u\in\Wrm^{1,p}(\Omega)$, there exists a gradient $p$-Young measure $(\nu_x)_{x\in\Omega}$ for which (i)-(iii) hold for almost every $x\in\Omega$ (with $M$ replaced by $M(x)$ and $\nu$ replaced by $\nu_x$).
\end{definition}

\begin{remark}\label{rk:remark_qc_hull}
\begin{enumerate}
\item In~\cite{KoRiWi13OPYM} it is shown that for the function $R(A)= - \det\, A$, $\mathbb{R}^{d\times d}$ is tightly contained in the $p$-quasiconvex hull of $S_R$. 

\item The \term{closed $p$-quasiconvex hull} of a set $S$, denoted $S^{p\text{-}qc}$, is classically defined as the set of all $M$ for which there exists a homogeneous gradient $p$-Young measure $\nu$ so that~(i),~(ii) hold in the above definition with $S$ in place of $S_{R(x,\frarg)}$.

\item Note that in the case that $R(x,\frarg)$ \emph{quasiconvex} (see~\eqref{eq:quasiconvex}) and $p\geq q$, the closed $p$-quasiconvex hull of $S_{R(x,\frarg)}$ is $S_{R(x,\frarg)}$ itself: Let $M \in S_{R(x,\frarg)}^{p\text{-}qc}$. By the Jensen-type inequality in the Kinderlehrer--Pedregal characterization of gradient Young measures, we obtain
\[
  \qquad R(x,M) \leq \lrangle{R(x,\frarg),\nu} \leq 0,
\]
thus $M \in S_{R(x,\frarg)}$. In particular, in this case, no strict superset of $S_{R(x,\frarg)}$ can satisfy conditions (i) and (ii) of Definition~\ref{def:p-full} and hence cannot be contained tightly in the $p$-quasiconvex hull of $S_{R(x,\frarg)}$.

\item Also, note that by (iii) we infer
\begin{equation*} 
  \qquad \int |A|^p \dd \nu(A)\leq C \biggl[ \int |A-M|^p \dd \nu(A) +|M|^p \biggr]
\leq C \bigl[ |R(x,M)|^{p/q}+|M|^p \bigr].
\end{equation*}
\end{enumerate}
\end{remark}

Recall that, by the characterization of Kinderlehrer--Pedregal~\cite{KinPed91CYMG,KinPed94GYMG}, $\nu=(\nu_x)_{x\in\Omega}$ is a \term{gradient $p$-Young measure}, that is, it is generated by a sequence of $\Lrm^p$-norm-bounded gradients, if and only if the following conditions hold:
\begin{itemize}
\item[(I)] $\displaystyle\int_{\Omega}\int\abs{A}^p \dd \nu_x(A)<\infty$;
\item[(II)] the barycenter $[\nu](x) := \int A \dd \nu_x(A)$ is a gradient, i.e.\ there exists $\nabla u \in \Lrm^p(\Omega;\R^{d \times d})$ with $[\nu] = \nabla u$ a.e.;
\item[(III)] for every quasiconvex function $h \colon \R^{d \times d} \to \R$ with $\abs{h(A)}\leq C(1+\abs{A}^p)$, the Jensen-type inequality
\[
  \qquad h(\nabla u(x)) \leq  \int h(A) \dd \nu_x(A)  \qquad\text{holds for a.e.\ $x \in \Omega$.}
\]
\end{itemize}
We also introduce, for $R\in\Rcal^{p,q}(\Omega;\R^{d \times d})$, the following pointwise condition, expressing that $\nu$ satisfies our side constraint:
\begin{itemize}
\item[(IV)] $\supp{\nu_x} \subset S_{R(x,\frarg)}$ for a.e.\ $x \in \Omega$.
\end{itemize}
Our abstract characterization result for gradient $p$-Young measures is the following:

\begin{theorem} \label{thm:main_abstract}
Let $\Omega \subset \R^d$ be open and bounded with $|\partial\Omega|=0$, let $R\in\Rcal^{p,q}(\Omega;\R^{d \times d})$ for $1 \leq q < \infty$, $1 < p < \infty$, and suppose that $\R^{d \times d}$ is tightly contained in the $p$-quasiconvex hull of $(S_{R(x,\frarg)})_{x\in\Omega}$, uniformly in $x$. Then, the following are equivalent for a $p$-Young measure $\nu = (\nu_x)_{x\in\Omega} \subset \Mbf^1(\R^{d \times d})$:
\begin{itemize}
  \item[(i)] There exists a sequence of gradients $(\nabla u_j) \subset \Lrm^p(\Omega;\R^{d \times d})$ that generates $\nu$, such that 
\[
  \qquad \nabla u_j(x) \in S_{R(x,\frarg)} \quad\text{for all $j\in\N$ and a.e.~$x\in\Omega$.}
\]
  \item[(ii)] The conditions (I)--(IV) hold.
\end{itemize}
Furthermore, in this case the sequence $(u_j)$ can be chosen such that $(\nabla u_j)$ is $p$-equiintegrable and $u_j - u\in\Wrm^{1,p}_0(\Omega,\mathbb{R}^d)$ where $u\in\Wrm^{1,p}(\Omega,\R^d)$ is the deformation underlying $\nu$ (i.e.\ the function whose gradient is the barycenter of $\nu$).
\end{theorem}

We first prove a key proposition, representing a convergence principle in the spirit of convex integration.

\begin{proposition}\label{prop:convexint}
Assume that $\Omega$, $R$, $p$, $q$ are as in the preceding theorem and suppose that $\R^{d \times d}$ is contained tightly in the $p$-quasiconvex hull of $(S_{R(x,\frarg)})_{x\in\Omega}$ uniformly in $x$, and let $u\in \Wrm^{1,p}(\Omega;\R^d)$. Then there exists $v\in \Wrm^{1,p}(\Omega;\R^d)$ such that
\begin{itemize}
\item[(i)]$\nabla v(x) \in S_{R(x,\frarg)}$\quad for a.e.\ $x \in \Omega$,\vspace{0.2cm}
\item[(ii)] $v - u\in\Wrm^{1,p}_0(\Omega,\mathbb{R}^d)$,\vspace{0.2cm}
\item[(iii)]$\displaystyle \norm{\nabla v-\nabla u}^p_p\leq C\int_{\Omega} \ONE_{\{y\,:\,R(y,\nabla u(y))>0\}}(x) \, |R(x,\nabla u(x))|^{p/q}\dd x,$\vspace{0.2cm}
\end{itemize}
where $C>0$ is a constant independent of $u$.
\end{proposition}

\begin{remark}
The preceding theorem and proposition also hold in the more general situation where a family $(D_x)_{x\in\Omega}$ is tightly contained in the $p$-quasiconvex hull of $S_{R(x,\frarg)}$ uniformly in $x$ (note that Definition~\ref{def:p-full} can be suitably generalized to $x$-dependent sets $D_x$ in an obvious way), under the additional assumption that any gradient $p$-Young measure $\nu$ with $\supp \nu_x \subset D_x$ a.e.~can be generated by a $p$-equiintegrable sequence of gradients $(\nabla u_j)$ such that $\nabla u_j(x)\in D_x$ a.e. and $u_j - u\in\Wrm^{1,p}_0(\Omega,\R^d)$ ($u$ denotes the map underlying $\nu$).
\end{remark}

\begin{proof}[Proof of Proposition~\ref{prop:convexint}]
Assume without loss of generality that
\[
  \int_{\Omega}\ONE_{\{y\,:\,R(y,\nabla u(y))>0\}}(x) \, |R(x,\nabla u(x))|^{p/q}\dd x > 0.
\]
We construct a sequence of functions $\{v^l\}_{l\in\N}$, bounded in $\Wrm^{1,p}(\Omega;\R^d)$, such that 
\begin{align}
\label{eq:det_proof}
&v^l - u \in\Wrm^{1,p}_0(\Omega;\R^d),\\
\label{eq:smalldet}
&\int_{\Omega}\ONE_{\{R(y,\nabla v^l(y))>0\}}(x) |R(x,\nabla v^l(x))|^{p/q}\dd x \\
&\qquad \leq 2^{-lp}\int_{\Omega}\ONE_{\{R(y,\nabla u(y))>0\}}(x) |R(x,\nabla u(x))|^{p/q}\dd x, \notag\\
\label{eq:cauchy}
&\int_{\Omega}|\nabla v^{l+1}(x)-\nabla v^l(x)|^p\dd x \\
&\qquad \leq 2^{-(l-1)p}C\int_{\Omega}\ONE_{\{R(y,\nabla u(y))>0\}}(x) |R(x,\nabla u(x))|^{p/q}\dd x, \notag
\end{align}
where $C>0$ is a constant.

Let us construct the sequence inductively. Set $v^0 = u$ so that~\eqref{eq:det_proof} and~\eqref{eq:smalldet} are satisfied. If $v^l\in \Wrm^{1,p}(\Omega;\R^d)$ has been constructed to satisfy~\eqref{eq:det_proof} and~\eqref{eq:smalldet}, we find $v^{l+1}$ in the following way: since $\R^{d \times d}$ is tightly contained in the $p$-quasiconvex hull of $S_{R(x,\frarg)}$ uniformly in $x$, there exists a gradient $p$-Young measure $(\nu_x^l)_{x\in\Omega}$ with $[\nu_x^l]=\nabla v^l(x)$ and with support in the set $S_{R(x,\frarg)}$ almost everywhere. 
Observe that by~(iii) in Definition~\ref{def:p-full}, for $x \in \Omega$ such that $R(x,\nabla v^l(x))\leq 0$ we have $\nu_x^l=\delta_{\nabla v^l(x)}$.

By standard Young measure arguments (see for example~\cite{Pedr97PMVP}), there exists a $p$-equiintegrable sequence of gradients $(\nabla v^{l,m})_m \subset \Lrm^p(\Omega;\R^{d \times d})$ generating $\nu^l$ such that $v^{l,m} - v^l\in\Wrm^{1,p}_0(\Omega;\R^d)$ and hence $v^{l,m} - u\in\Wrm^{1,p}_0(\Omega;\R^d)$ for all $m\in\mathbb{N}$.

We define $g:\Omega\times\R^{d\times d}\to\R$ by
\begin{equation}\label{eq:dettest}
g(x,A)= \ONE_{\{y\,:\,R(y,A)>0\}}(x) |R(x,A)|^{p/q}=\begin{cases}R(x,A)^{p/q} & \text{if $R(x,A)>0$,}\\
0 & \text{otherwise.}
\end{cases}
\end{equation}
Using $g$ as a test function and the fact that $\nu^l_x$ is supported in $S_{R(x,\frarg)}$, by Young measure representation, we may choose $m$ large enough, say $m=M$, and define $\nabla v^{l+1}:=\nabla v^{l,M}$ such that
\begin{align*}
&\int_{\Omega}\ONE_{\{R(y,\nabla v^{l+1}(y))>0\}}(x)|R(x,\nabla v^{l+1}(x))|^{p/q}\dd x \\
&\qquad \leq 2^{-(l+1)p}\int_{\Omega}\ONE_{\{R(y,\nabla u(y))>0\}}(x) |R(x,\nabla u(x))|^{p/q}\dd x,
\end{align*}
i.e.~\eqref{eq:det_proof} as well as \eqref{eq:smalldet} hold for $v^{l+1}$.
Also, by taking $M$ even larger if necessary, we can ensure that also
\begin{equation}\label{eq:Mchoice}
\int_{\Omega}|\nabla v^{l+1}(x)-\nabla v^l(x)|^p\dd x\leq 2^p\int_{\Omega}\int|A-\nabla v^l(x)|^p\dd\nu_x^l(A) \dd x
\end{equation}
Indeed, this follows again from Young measure representation for the integrand $|A-\nabla v^l(x)|^p$. 

Next, for any $l \in \N$, by property (iii) of Definition~\ref{def:p-full} and~\eqref{eq:smalldet} we infer that 
\[
\begin{aligned}
\int_{\Omega}\int|A-\nabla v^l(x)|^p\dd\nu_x^l(A)\dd x&\leq C\int_{\Omega}\ONE_{\{R(y,\nabla u(y))>0\}}(x) |R(x,\nabla v^l(x))|^{p/q}\dd x\\
&\leq 2^{-lp} C \int_{\Omega} \ONE_{\{R(y,\nabla u(y))>0\}}(x) |R(x,\nabla u(x))|^{p/q}\dd x
\end{aligned}
\] 
for a constant $C>0$ independent of $x$. Combining with~\eqref{eq:Mchoice} we get the estimate
\begin{equation*}
\int_{\Omega}|\nabla v^{l+1}(x)-\nabla v^l(x)|^p\dd x\leq 2^{-(l-1)p} C \int\ONE_{\{R(y,\nabla u(y))>0\}}(x) |R(x,\nabla u(x))|^{p/q}\dd x,
\end{equation*}
which is \eqref{eq:cauchy}, completing the definition of our sequence.

The result then follows readily: by \eqref{eq:cauchy}, $(\nabla v^l)_{l\in\N}$ is a Cauchy sequence in $\Lrm^p(\Omega;\R^{d \times d})$ and therefore has a strong $\Lrm^p$-limit $\nabla v$. In particular, it holds that $v - u\in\Wrm^{1,p}_0(\Omega;\R^d)$ and (ii) follows. Using the triangle inequality and~\eqref{eq:cauchy}, we deduce that
\[
\begin{aligned}
\norm{\nabla v-\nabla u}_p&\leq\sum_{l=0}^{\infty}\norm{\nabla v^{l+1}-\nabla v^l}_p\\
&\leq C^{1/p}\left(\int_{\Omega}\ONE_{\{R(y,\nabla u(y))>0\}}(x) |R(x,\nabla u(x))|^{p/q}\dd x\right)^{1/p}\sum_{l=0}^{\infty}2^{-(l-1)}\\
&\leq 4C^{1/p}\left(\int_{\Omega}\ONE_{\{R(y,\nabla u(y))>0\}}(x) |R(x,\nabla u(x))|^{p/q}\dd x\right)^{1/p},
\end{aligned}
\] 
proving (iii). Lastly, $(\nabla v^l)_l$ is $p$-equiintegrable (being Cauchy in $\Lrm^p$), and since $|R(x,\nabla v^l(x))|^{p/q}\leq C(|\kappa(x)|^{p/q}+|\nabla v^l(x)|^p)$, also $\{|R(\frarg,\nabla v^l)|^{p/q}\}_{l\in\N}$ is equiintegrable and converges, up to a subsequence, to $|R(\frarg,\nabla v)|^{p/q}$. Therefore, by Vitali's Convergence Theorem, 
\[
\int_{\Omega}\ONE_{\{R(y,\nabla v(y))>0\}}(x)|R(x,\nabla v(x))|^{p/q}\dd x=0,
\] 
which implies $R(x,\nabla v(x))\leq 0$ for a.e.\ $x\in\Omega$, i.e.\ (i), and the proof is complete.
\end{proof}

\begin{proof}[Proof of Theorem~\ref{thm:main_abstract}]
(i) $\Rightarrow$ (ii):
Conditions (I)--(III) follow from the usual Kinderlehrer--Pedregal Theorem in~\cite{KinPed94GYMG}. Regarding (IV), let $h\in \Lrm^{\infty}(\Omega\times\R^{d\times d})$ be Carath\'{e}odory and such that $\supp h(x,\frarg) \subset\subset \R^{d \times d} \setminus S_{R(x,\frarg)}$ for almost every $x$. Then, by the assumptions on $\nabla u_j$,
\[
  \int_\Omega \int h(x,A) \dd \nu_x(A) \dd x = \lim_{j\to\infty} \int_\Omega  h(x,\nabla u_j(x)) \dd x = 0.
\]
Varying $h$, we infer that $\supp \nu_x \subset S_{R(x,\frarg)}$ for a.e.\ $x \in \Omega$.

(ii) $\Rightarrow$ (i):
For $1 < p < \infty$, $1 \leq q < \infty$ as in Definition~\ref{def:p-full}, let $\nu$ be a gradient $p$-Young measure with $\supp \nu_x \subset S_{R(x,\frarg)}$ for a.e.\ $x \in \Omega$. Standard results yield that there exists a generating sequence $(\nabla u_j)$ for $\nu$ which is $p$-equiintegrable and satisfies $u_j - u\in\Wrm^{1,p}_0(\Omega;\R^d)$ where $\nabla u=[\nu]$. By Young measure representation applied to the test function $g$ in~\eqref{eq:dettest} and the assumption on the support of $\nu$, we may assume (after passing to a subsequence if necessary) that
\begin{equation}\label{eq:closetoK}
\int_{\Omega}\ONE_{\{R(y,\nabla u_j(y))>0\}}(x)|R(x,\nabla u_j(x))|^{p/q}\dd x<\frac{1}{j^p}.
\end{equation}
Applying Proposition~\ref{prop:convexint} to each $u_j$, we obtain a new sequence $(v_j)$, such that $\nabla v_j(x)\in S_{R(x,\frarg)}$ a.e., $v_j - u\in\Wrm^{1,p}_0(\Omega;\R^d)$ and, by~\eqref{eq:closetoK} and part~(iii) of Proposition~\ref{prop:convexint},
\[
\norm{\nabla u_j-\nabla v_j}_p<\frac{C^{1/p}}{j}.
\]  
Hence $(\nabla v_j)$ is $p$-equiintegrable and generates $\nu$.
\end{proof}

\section{Differential inclusions involving prescribed Jacobians} \label{sc:geometry}

In this section we show that all of $\R^{d\times d}$ is tightly contained in the $p$-quasiconvex hull of $(S_{R(x,\frarg)})_{x \in \Omega}$ uniformly in $x$ for all $p\in[1,d)$, where for $J_1$ and $J_2$ as in Theorem~\ref{thm:main_intro},
\[
R(x,A) := \max\{J_1(x) - \det\, A, \det\, A - J_2(x), 0 \}
\]
and the corresponding sublevel sets are given by
\[
S_{R(x,\frarg)}=\setb{A\in\R^{d\times d}}{J_1(x)\leq \det\, A \leq J_2(x)}.
\]

Then Theorem~\ref{thm:main_abstract} establishes Theorem~\ref{thm:main_intro}.
We note that the above function $R$ is indeed an element of $\Rcal^{p,d}(\Omega,\R^{d\times d})$. To see this, note that
\begin{align*}
0 \leq R(x,A) &\leq \max\{J_1^+(x) - \det\, A, \det\, A + J_2^-(x), 0 \} \\
&\leq J_1^+(x) + J_2^-(x) + C|A|^d =: \kappa(x) + C|A|^d
\end{align*}
with $\kappa\in \Lrm^{p/d}(\Omega)$. Of course, since $R(x,\frarg)$ is quasiconvex for a.e.~$x\in\Omega$, by Remark~\ref{rk:remark_qc_hull} (3) we are forced to restrict attention to $p<d$.

The fact that $\R^{d\times d}$ is contained tightly in the $p$-quasiconvex hull of the above $(S_{R(x,\frarg)})_{x\in\Omega}$ will be a corollary to the following proposition, which vastly generalizes Proposition 3.1 in~\cite{KoRiWi13OPYM}:

\begin{proposition} \label{prop:geometry}
Let $1 \leq p < d$, $r\in\Lrm^{p/d}(\Omega)$, and set $R(x,A)=|\det\, A-r(x)|$. Then $R\in \Rcal^{p,d}(\Omega,\R^{d\times d})$ and $\R^{d\times d}$ is tightly contained in the $p$-quasiconvex hull of $(S_{R(x,\frarg)})_{x\in\Omega}$ uniformly in $x$.

\end{proposition}

Before proving Proposition~\ref{prop:geometry}, we state and prove in the form of a corollary the result concerning the function $R(x,A)=\max\{J_1(x) - \det\, A, \det\, A - J_2(x), 0\}$.

\begin{corollary}\label{cor:geometry}
Let $1 \leq p < d$ and $J_1$, $J_2$ as in Theorem~\ref{thm:main_intro}. Set
\[
R(x,A)=\max\bigl\{J_1(x) - \det\, A, \det\, A - J_2(x), 0\bigr\}.
\]
Then $\R^{d\times d}$ is tightly contained in the $p$-quasiconvex hull of $(S_{R(x,\frarg)})_{x\in\Omega}$ uniformly in $x$.
\end{corollary}

\begin{proof}
Suppose $M=\nabla u$ for some $u\in\Wrm^{1,p}(\Omega;\R^d)$. Define the function $r_M:\Omega\to\Rbb$ in the following way:
\[
r_M(x)=\begin{cases}\det\, M(x) & \text{if $J_1(x)\leq\det\, M(x)\leq J_2(x)$,}\\
J_1(x) & \text{if $\det\, M(x)<J_1(x)$,}\\
J_2(x) & \text{if $\det\, M(x)>J_2(x)$.}
\end{cases}
\] 
It then follows from the assumptions on $J_1$ and $J_2$ that $r_M\in\Lrm^{p/d}(\Omega)$, and therefore Proposition~\ref{prop:geometry} applied to $r_M$ yields a $p$-gradient Young measure $(\nu_x)_x$ such that, for almost every $x\in\Omega$, 
\begin{itemize}
\item[(i)] $[\nu_x] = M(x)$;
\item[(ii)] $\supp \nu_x \subset \setb{A\in\R^{d\times d}}{\det\, A = r_M(x)}\subset S_{R(x,\frarg)}$;
\item[(iii)] $\displaystyle\int |A - M|^p\dd \nu_x(A) \leq C|r_M(x) - \det\, M(x)|^{p/d}$,
\end{itemize}
where $C$ is independent of $M$ and $x$. The claim now follows from the observation, using the definitions of $r_M$ and $R$, that
\[
|r_M(x) - \det\, M(x)|\leq R(x,M(x))
\]
for almost every $x$.
\end{proof}

\subsection{Three dimensions}

We first prove Proposition~\ref{prop:geometry} for $d=3$ only. The proof for $d=2$ is similar but simpler and the proof for $d>3$ is outlined in the next section. 

\begin{proof}[Proof of Proposition~\ref{prop:geometry} for $d = 3$]
In the first steps of the proof, we fix a matrix $M_0$ and a real number $r$.

\proofstep{Step~1.}
Following~\cite{KoRiWi13OPYM}, we transform an arbitrary matrix $M_0$ to diagonal form using the real singular value decomposition and write $M_0 = \tilde{P} \tilde{D}_0 \tilde{Q}^T$ where $\tilde{D}_0 = \mathrm{diag}(\sigma_1,\sigma_2,\sigma_3)$ with $0\leq\sigma_1 \leq \sigma_2\leq\sigma_3$, and $\tilde{P},\tilde{Q}$ orthogonal matrices. If $\det\, M_0 < 0$, either $\tilde{P}$ or $\tilde{Q}$ has negative determinant, say $\det\, \tilde{P} < 0$ (the other case is similar). Then, $M_0 = P D_0 Q^T$ where
\[
  D_0 := \mathrm{diag}(\sigma_1,\sigma_2,-\sigma_3),  \qquad
  P := \tilde{P}\mathrm{diag}(1,1,-1),  \qquad
  Q := \tilde{Q},
\]
with $P,Q \in \mathrm{SO}(3)$ and $\det\, D_0 < 0$. Similarly, if $\det\, M_0 \geq 0$, we may write $M_0 = P D_0 Q^T$, where $P,Q \in \mathrm{SO}(3)$ and $\det\, D_0 \geq 0$ for
\[
  D_0 = \mathrm{diag}(\sigma_1,\sigma_2,\sigma_3).
\]
Note that if $D_0$ can be written as a laminate then the same holds for $M_0$ since $P (a \otimes b) Q^T = (Pa) \otimes (Qb)$ for any $a,b \in \R^3$. Also, we remark that the matrices $D_0$ and $M_0$ share the same
determinant and (Frobenius) matrix norm. Consequently, we may henceforth assume without loss of generality that
\[
  M_0 = \mathrm{diag}(\sigma_1,\sigma_2,\pm\sigma_3).
\]
with $0\leq\sigma_1\leq\sigma_2\leq\sigma_3$.

We now distinguish the cases $\sigma_3\geq\bigl(\frac{|r|}{2}\bigr)^{1/3}$ and $\sigma_3<\bigl(\frac{|r|}{2}\bigr)^{1/3}$.\\

\noindent\textbf{Case~I: $\sigma_3\geq\bigl(\frac{|r|}{2}\bigr)^{1/3}.$}

\proofstep{Step~I.2.}
Set $\gamma := \frac{|r - \det\, M_0|^{1/2}}{\sigma_3^{1/2}}$ and decompose $M_0$ twice along rank-one lines:
\begin{align*}
  M_0 &= \frac{1}{4} \bigl[ M_0 + \gamma (\ee_1 \otimes \ee_2) + \gamma (\ee_2 \otimes \ee_1) \bigr]
    + \frac{1}{4} \bigl[ M_0 + \gamma (\ee_1 \otimes \ee_2) - \gamma (\ee_2 \otimes \ee_1) \bigr] \\
  &\qquad + \frac{1}{4} \bigl[ M_0 - \gamma (\ee_1 \otimes \ee_2) + \gamma (\ee_2 \otimes \ee_1) \bigr]
    + \frac{1}{4} \bigl[ M_0 - \gamma (\ee_1 \otimes \ee_2) - \gamma (\ee_2 \otimes \ee_1) \bigr]. 
\end{align*}
Direct computation yields that two of these four matrices (either those where both $\gamma$'s come with the same sign, or those where the $\gamma$'s have different signs, depending on the sign of $\sigma_3$) have determinant $r$, and the other two have determinant $2\det\, M_0 - r$. We call the former ones $M_{1,G1}$ and $M_{1,G2}$ ($G$ for \textit{good}) and the latter ones $M_{1,B1}$ and $M_{1,B2}$ ($B$ for \textit{bad}), so that we have the decomposition
\begin{align*}
  M_0 = \frac{1}{4} M_{1,B1} + \frac{1}{4} M_{1,G1} + \frac{1}{4} M_{1,G2} + \frac{1}{4} M_{1,B2}
\end{align*}
with
\begin{align*}
  \det\, M_{1,G1} &= \det\, M_{1,G2} = r,  \\
  \det\, M_{1,B1} &= \det\, M_{1,B2} = 2\det\, M_0 - r. 
\end{align*}
Now, if $|\det\, M_{0}|\leq |r|$, it holds that $\frac{|r|}{2}\geq\frac{1}{4}|r-\det\, M_{0}|$ and, taking into account that $\sigma_3\geq\bigl(\frac{|r|}{2}\bigr)^{1/3}$, 
\begin{equation*}
\sigma_3\geq 4^{-1/3}|r-\det\, M_{0}|^{1/3}.
\end{equation*}
On the other hand, if $|\det\, M_{0}|> |r|$, we can use the fact that $\sigma_3\geq|\det\, M_{0}|^{1/3}$ to infer that
\begin{equation*}
\sigma_3\geq|\det\, M_{0}|^{1/3}\geq 2^{-1/3}(|r|+|\det\, M_{0}|)^{1/3}\geq 2^{-1/3}|r-\det\, M_{0}|^{1/3}.
\end{equation*}
This implies that there is a constant $C>0$, independent of $M_{0}$ and $r$, such that in either case 
\begin{equation*}
\sigma_3\geq C|r-\det\, M_{0}|^{1/3}.
\end{equation*}

It then follows that for $J = G1, G2, B1, B2$,
\begin{equation}  \label{eq:M_1X_dist}
  \abs{M_{1,J}-M_0} =  \biggl( 2 \frac{\abs{r - \det\, M_0}}{\sigma_3} \biggr)^{1/2} \leq C \frac{\abs{r - \det\, M_0}^{1/2}}{\abs{r - \det\, M_0}^{1/6}} = C \abs{r - \det\, M_0}^{1/3}
\end{equation}
for a constant $C > 0$ independent of $M_0$ and $r$, and also
\begin{equation}  \label{eq:rdet_dist}
  \abs{r - \det\, M_{1,J}} \leq  2 \abs{r - \det\, M_0}.
\end{equation}

Moreover, the singular value $\sigma_3$ is not altered by this construction, and so there is still a singular value of $M_{1,B1}$ and $M_{1,B2}$ with modulus at least $\bigl(\frac{|r|}{2}\bigr)^{1/3}$. Therefore we may recursively apply the procedure from the preceding steps to decompose $M_{1,B1}$ and $M_{1,B2}$ in turn taking the role of $M_0$. This yields matrices $M_{2,G1},\ldots,M_{2,G4}$, $M_{2,B1},\ldots,M_{2,B4}$ such that
\begin{align*}
  M_{1,B1} = \frac{1}{4} M_{2,G1} + \frac{1}{4} M_{2,G2} + \frac{1}{4} M_{2,B1} + \frac{1}{4} M_{2,B2}, \\
  M_{1,B2} = \frac{1}{4} M_{2,G3} + \frac{1}{4} M_{2,G4} + \frac{1}{4} M_{2,B3} + \frac{1}{4} M_{2,B4},
\end{align*}
and so on.

The laminate which we get after $k$ steps is then given by
\[
\nu_k:= \sum_{i=1}^{k}\sum_{j=1}^{2^i}\frac{1}{4^i}\delta_{M_{i,Gj}} + \sum_{j=1}^{2^{k}}\frac{1}{4^{k}}\delta_{M_{k,Bj}},
\]
where, for all $i$, $j$, $\det\, M_{i,Gj}=r$.

\proofstep{Step~I.3.}
It is clear that each $\nu_k$ satisfies $[\nu_k]=M_0$ and we turn attention to the distance integral in~(iii) of Definition~\ref{def:p-full}. That is,
\begin{equation*}
  \int \abs{A-M_0}^p \dd \nu_k(A) = \sum_{i=1}^{k} \sum_{j=1}^{2^i} \frac{1}{4^i} \abs{M_{i,Gj}-M_0}^p + \sum_{j=1}^{2^{k}} \frac{1}{4^{k}} \abs{M_{k,Bj}-M_0}^p.
\end{equation*}
Let us define $X_i := M_{i,Gj}$, $X_0 := M_0$, and $X_{\ell-1}$ to be the matrix $M_{\ell-1,Bj}$ with $j \in \{1,\ldots,2^{\ell-1}\}$ such that $X_\ell$ is constructed from $X_{\ell-1}$ by laminating as in the previous proof step (with the convention $M_{0,B1} := M_0$); similarly, let $Y_{k} := M_{k,Bj}$, $Y_0 := M_0$, and $Y_{\ell-1}$ defined analogously to $X_{\ell-1}$. Then, $\sum_{\ell=1}^i X_\ell - X_{\ell-1} = M_{i,Gj}-M_0$ and $\sum_{\ell=1}^{k} Y_\ell - Y_{\ell-1} = M_{k,Bj}-M_0$, and by the virtue of the triangle inequality
\begin{equation*}
  \int \abs{A-M_0}^p \dd \nu_k(A) \leq \sum_{i=1}^{k} \sum_{j=1}^{2^i} \frac{1}{4^i} \biggl( \sum_{\ell=1}^i \abs{X_\ell -X_{\ell-1}} \biggr)^p + \sum_{j=1}^{2^{k}} \frac{1}{4^{k}} \biggl( \sum_{\ell=1}^{k} \abs{Y_\ell-Y_{\ell-1}} \biggr)^p.
\end{equation*}

In order to get bounds on $\abs{X_\ell - X_{\ell-1}}$, we use~\eqref{eq:M_1X_dist} and then~\eqref{eq:rdet_dist} recursively.
Thus,
\begin{align*}
  \sum_{\ell=1}^i \abs{X_\ell -X_{\ell-1}} &\leq \sum_{\ell=1}^i C \, \abs{r - \det\, X_{\ell-1}}^{1/3} \leq \sum_{\ell=1}^i C 2^{(\ell-1)/3} \, \abs{r - \det\, M_0}^{1/3} \\
  &\leq \frac{C \, \abs{r - \det\, M_0}^{1/3}}{2^{1/3}-1} 2^{i/3}
\end{align*}
and a similar estimate holds for the sum involving the $Y_\ell$'s with $i$ replaced by $k$. Hence, 
\begin{align}
 \int \abs{A-M_0}^p \dd \nu_k(A) &\leq C \abs{r - \det\, M_0}^{p/3} \left[ \sum_{i=1}^{k} (2^{p/3 - 1})^i + (2^{p/3 - 1})^{k}  \right]\nonumber \\
  &\leq C  \abs{r - \det\, M_0}^{p/3} \biggl[ \frac{1}{1-\rho} + \rho^{k} \biggr] \nonumber\\
  &\leq C_p \abs{r - \det\, M_0}^{p/3},
  \label{eq:geom_prop1}
\end{align}
where $\rho:=2^{p/3 - 1}<1$ (since $p < 3$) and $C_p>0$ is a constant depending on $p$ (but not on $M_0$ or $r$) which blows up as $p\to 3$.
Also, each $\nu_k$ is a probability measure and by \eqref{eq:geom_prop1} we deduce that
\begin{align}
 \int \abs{A}^p \dd \nu_k(A) & \leq 2^p \left[\int \abs{A-M_0}^p \dd \nu_k(A) + \abs{M_0}^p \right]\nonumber\\
  & \leq 2^p C_p \abs{r - \det\, M_0}^{p/3} + 2^p \abs{M_0}^p.
  \label{eq:geom_prop2}
\end{align}

Observe moreover that the mass of $\nu_k$ carried by the matrices outside $S_R$ is
\begin{equation}\label{eq:mass}
 \nu_k \bigl( \setb{ A \in \R^{3 \times 3} }{ \det\, A \neq r } \bigr) = \frac{2^{k}}{4^{k}} \to 0  \qquad\text{as $k\to\infty$} 
\end{equation}

\noindent\textbf{Case~II: $\sigma_3<\bigl(\frac{|r|}{2}\bigr)^{1/3}$.}

\proofstep{Step~II.2.} Again, we assume that $M_0$ is given by
\[
  M_0 := \mathrm{diag}(\sigma_1,\sigma_2,\pm\sigma_3)
\]
with $0<\sigma_1\leq\sigma_2\leq\sigma_3$, but now $\sigma_3<\bigl(\frac{|r|}{2}\bigr)^{1/3}$. We decompose $M_0$ along a rank-one line as
\[
M_0=\frac{1}{2}\left[M_0+\delta(e_3\otimes e_3)\right]+\frac{1}{2}\left[M_0-\delta(e_3\otimes e_3)\right]
=:\frac{1}{2}M_0^++\frac{1}{2}M_0^-,
\]
where we choose $\delta=2\bigl(\frac{|r|}{2}\bigr)^{1/3}$. Then, the singular values $\sigma_3+\delta$ and $\sigma_3-\delta$ of $M_0^+$ and $M_0^-$, respectively, have absolute value at least $\bigl(\frac{|r|}{2}\bigr)^{1/3}$. Moreover, we have the estimates
\begin{align}
\abs{M_0^{\pm} - M_0} &= 2\biggl(\frac{|r|}{2}\biggr)^{1/3}\leq 2|r-\det\, M_0|^{1/3} \label{eq:dist+-},\\
\abs{r - \det\, M_0^{\pm}} &\leq 3\abs{r - \det\, M_0}. \label{eq:rdet_dist+-}
\end{align}
Indeed, both inequalities follow from the observation that $|\det\, M_0|\leq\sigma_3^3<|r|/2$, and therefore $|r-\det\, M_0|>|r|/2$.

\proofstep{Step~II.3.}
We can treat $M_0^{\pm}$ exactly as in Case~I, which is now applicable to $M_0^+$ and $M_0^-$. This gives us two sequences of laminates $\nu_k^+$ and $\nu_k^-$ with $[\nu_k^{\pm}]=M_0^{\pm}$, 
\begin{equation}\label{eq:mass2}
\nu^\pm_k \bigl( \setb{ A \in \R^{3 \times 3} }{ \det\, A \neq r } \bigr) \to 0
\end{equation}
as $k\to\infty$, and the estimate
\begin{equation}\label{eq:pmestimate}
\int|A-M_0^{\pm}|^p\dd\nu_k^{\pm}(A)\leq C_p|r-\det\, M_0^{\pm}|^{p/3}
\end{equation} 
for $1\leq p<3$, where $C_p$ does not depend on $k$, $M_0$ or $r$. It follows that the measure $\nu_k$ defined by $\nu_k=\frac{1}{2}\nu_k^++\frac{1}{2}\nu_k^-$ satisfies $[\nu_k]=M_0$ and $\nu_k \bigl( \setb{ A \in \R^{3 \times 3} }{ \det\, A \neq r } \bigr) \to 0
$. Moreover, combining~\eqref{eq:pmestimate} with~\eqref{eq:dist+-} and~\eqref{eq:rdet_dist+-}, we have 
\begin{align}
\int|A-M_0|^p\dd\nu_k(A)&= \frac{1}{2}\int|A-M_0|^p\dd\nu_k^+(A)+\frac{1}{2}\int|A-M_0|^p\dd\nu_k^-(A) \nonumber\\
&\leq C_p\biggl[|M_0^+-M_0|^p+|M_0^--M_0|^p \nonumber\\
& \qquad\qquad +\int |A-M_0^+|^p\dd\nu_k^+(A) +\int |A-M_0^-|^p\dd\nu_k^-(A)\biggr] \nonumber\\
&\leq C_p\left[|r-\det\, M_0|^{p/3}+|r-\det\, M_0^{\pm}|^{p/3}\right] \nonumber\\
&\leq C_p|r-\det\, M_0|^{p/3}.\label{eq:pbounded2}
\end{align}

\proofstep{Step~4.} In this last step, let $M(x)=\nabla u(x)$ for some $u\in\Wrm^{1,p}(\Omega;\R^d)$ and $r\in\Lrm^{p/d}(\Omega)$. Applying the previous steps to $M_0=M(x)$ and $r=r(x)$ for almost every $x$, we obtain a sequence $((\nu_{x,k})_{x\in\Omega})_{k\in\Nbb}$ of parametrized probability measures. Note that, for every $k$, $(\nu_{x,k})_x$ is weakly* measurable. Indeed, $M:\Omega\to\Rbb^{d\times d}$ is measurable by assumption, and the matrices obtained from the rank-one splittings of $M_0$ in the previous steps depend continuously on $M_0$ (to be more precise, there is a discontinuity at $\sigma_3=(|r|/2)^{1/3}$, which is where Cases I and II bifurcate, thus rendering the dependence of the matrices $M_{i,Gj}$, $M_{i,B_j}$ on $M_0$ only \emph{piecewise} continuous). 

Moreover, by the bounds~\eqref{eq:geom_prop2} and~\eqref{eq:pbounded2}, we obtain that $(\nu_{x,k})_k$ is a bounded sequence in the space $\Lrm_w^{\infty}(\Omega;\Mbf^1(\Rbb^{d\times d}))$ of weakly* measurable maps from $\Omega$ into $\Mbf^1(\Rbb^{d\times d})$. Therefore (cf.~\cite{Mull99VMMP}, Sections 3.1 and 3.4), there exists a subsequence (not relabeled) and a Young measure $\nu=(\nu_x)_x\in\Lrm_w^{\infty}(\Omega;\Mbf^1(\Rbb^{d\times d}))$ such that
\begin{equation}\label{eq:YMconvergence}
\int_{\Omega}\int_{\Rbb^{d\times d}}f(x,A)\dd \nu_{x,k}(A)\dd x\to \int_{\Omega}\int_{\Rbb^{d\times d}}f(x,A)\dd \nu_x(A)\dd x
\end{equation}
as $k\to\infty$, for every Carath\'{e}odory function $f:\Omega\times\Rbb^{d\times d}\to\Rbb$ such that the family $\left(\int f(x,A)\dd\nu_{x,k}(A)\right)_k$ is equiintegrable. 

We claim that $(\nu_x)_x$ is a $p$-gradient Young measure that satisfies (i)-(iii) from Definition~\ref{def:p-full}, which then implies the proposition. First, a standard diagonal argument together with the bounds~\eqref{eq:geom_prop2} and~\eqref{eq:pbounded2} implies that indeed $\nu$ is a $p$-gradient Young measure.

Property (i) follows from the fact that $[\nu_{x,k}]=M(x)$ for almost every $x$, and from~\eqref{eq:YMconvergence} with $f(x,A)=\psi(x)A$ for any $\psi\in \Lrm^\infty(\Omega)$. Varying over $\psi$ then gives $[\nu_x]=M(x)$ almost everywhere.

Property (ii) is a consequence of~\eqref{eq:mass},~\eqref{eq:mass2} as well as the choice $f(x,A)=\psi(x)\ONE_{\{M\,:\,\det\, M\neq r(x)\}}(A)$ in~\eqref{eq:YMconvergence} (the characteristic function is lower semicontinuous with respect to $A$, which makes it admissible as a test function; see~\cite{Mull99VMMP}, Section 3.4).

Finally, (iii) is a consequence of~\eqref{eq:geom_prop2} and~\eqref{eq:pbounded2} in conjunction with~\eqref{eq:YMconvergence} using $f(x,A)=|A-M(x)|^p$. Notice that the equiintegrability of $\left(\int f(x,A)\dd\nu_{x,k}(A)\right)_k$ for this choice of $f$ follows from~\eqref{eq:geom_prop1} and~\eqref{eq:pbounded2} and the assumptions on $r$ and $M$. 

\end{proof}

\subsection{Arbitrary dimensions}
\label{sc:arbitrary_dim}

In this part, we briefly outline the proof of Proposition~\ref{prop:geometry} for arbitrary dimensions. The cases $d=3$ and $d>3$ are quite similar, so that we only provide the basic estimates, everything else remaining the same.
\begin{proof}[Proof of Proposition~\ref{prop:geometry} for $d>3$] \proofstep{Step 1.} As before we bring a matrix $M_0\in \R^{d\times d}$ into diagonal form and write
\[
  M_0 := \mathrm{diag}(\sigma_1,\sigma_2,\ldots,\pm\sigma_d),
\]
for which $0\leq\sigma_1\leq\sigma_2\leq\ldots\leq\sigma_d$.

We now distinguish the cases $\sigma_3\cdots\sigma_d\geq\bigl(\frac{|r|}{2}\bigr)^{(d-2)/d}$ and $\sigma_3\cdots\sigma_d<\bigl(\frac{|r|}{2}\bigr)^{(d-2)/d}$.\\

\noindent\textbf{Case~I: $\sigma_3\cdots\sigma_d\geq\bigl(\frac{|r|}{2}\bigr)^{(d-2)/d}$.}

\proofstep{Step~I.2.}
Set $\gamma := \frac{|r - \det\, M_0|^{1/2}}{(\sigma_3\cdots\sigma_d)^{1/2}}$ and decompose $M_0$ twice along rank-one lines:
\begin{align*}
  M_0  &= \frac{1}{4} \bigl[ M_0 + \gamma (\ee_1 \otimes \ee_2) + \gamma (\ee_2 \otimes \ee_1) \bigr]
    + \frac{1}{4} \bigl[ M_0 + \gamma (\ee_1 \otimes \ee_2) - \gamma (\ee_2 \otimes \ee_1) \bigr] \\
  &\qquad + \frac{1}{4} \bigl[ M_0 - \gamma (\ee_1 \otimes \ee_2) + \gamma (\ee_2 \otimes \ee_1) \bigr]
    + \frac{1}{4} \bigl[ M_0 - \gamma (\ee_1 \otimes \ee_2) - \gamma (\ee_2 \otimes \ee_1) \bigr] \\
  &=: \frac{1}{4} M_{1,B1} + \frac{1}{4} M_{1,G1} + \frac{1}{4} M_{1,G2} + \frac{1}{4} M_{1,B2},
\end{align*}
where the ``good'' and the ``bad'' matrices are again labeled such that
\begin{align*}
  \det\, M_{1,G1} &= \det\, M_{1,G2} = r,  \\
  \det\, M_{1,B1} &= \det\, M_{1,B2} = 2\det\, M_0 - r. 
\end{align*}

If $|\det\, M_{0}|\leq |r|$, it holds that $\frac{|r|}{2}\geq\frac{1}{4}|r-\det\, M_{0}|$ and, taking into account that $\sigma_3\cdots\sigma_d\geq \bigl(\frac{|r|}{2}\bigr)^{(d-2)/d}$, 
\[
\sigma_3\cdots\sigma_d\geq \left(\frac{1}{4}\right)^{(d-2)/d}|r-\det\, M_{0}|^{(d-2)/d}.
\]
If however $|\det\, M_{0}|> |r|$, we need a more involved estimate: Note that
\begin{equation}\label{eq:star_d}
(\sigma_1\sigma_2)^{d-2} \leq (\sigma_3\cdots\sigma_d)(\sigma_3\cdots\sigma_d)=(\sigma_3\cdots\sigma_d)^2.
\end{equation}
Then, through \eqref{eq:star_d}, we obtain that
\[
|\det\, M_{0}|^{(d-2)/d}\leq (\sigma_3\cdots\sigma_d)^{2/d}(\sigma_3\cdots\sigma_d)^{(d-2)/d}=\sigma_3\cdots\sigma_d,
\]
i.e.
\[
(\sigma_3\cdots\sigma_d)^{d/(d-2)}\geq|\det\, M_{0}|\geq \frac{1}{2}(|r|+|\det\, M_{0}|)\geq \frac{1}{2}|r-\det\, M_{0}|.
\]
This shows that there is a constant $C>0$, independent of $M_{0}$ and $r$, such that in either case 
\[
\sigma_3\cdots\sigma_d\geq C|r-\det\, M_{0}|^{(d-2)/d}.
\]
Then the following estimates hold:
\begin{align*} 
  \abs{M_{1,J}-M_0} &=  \biggl( 2 \frac{\abs{r - \det\, M_0}}{\sigma_3\cdots\sigma_d} \biggr)^{1/2} \leq C \frac{\abs{r - \det\, M_0}^{1/2}}{|r-\det\, M_0|^{(d-2)/2d}}= C\abs{r - \det\, M_0}^{1/d},\\
  \abs{r - \det\, M_{1,J}} &\leq  2 \abs{r - \det\, M_0}.
\end{align*}

We may now proceed as in the case $d=3$.

\noindent\textbf{Case~II: $\sigma_3\cdots\sigma_d<\bigl(\frac{|r|}{2}\bigr)^{(d-2)/d}$.}

\proofstep{Step~II.2.} We may still assume that $M_0$ is given by
\[
  M_0 := \mathrm{diag}(\sigma_1,\sigma_2,\ldots,\pm\sigma_d)
\]
with $0<\sigma_1\leq\sigma_2\leq\ldots\leq\sigma_d$, but now $\sigma_3\cdots\sigma_d<\bigl(\frac{|r|}{2}\bigr)^{(d-2)/d}$. We decompose $M_0$ along rank-one lines as
\[
\begin{aligned}
M_0&=\frac{1}{2}\left[M_0+\delta(e_3\otimes e_3)\right]+\frac{1}{2}\left[M_0-\delta(e_3\otimes e_3)\right]\\
&=:\frac{1}{2}M_0^++\frac{1}{2}M_0^-,
\end{aligned}
\]
where we choose $\delta=2\bigl(\frac{|r|}{2}\bigr)^{1/d}$. Then, the singular values $\sigma_3+\delta$ and $\sigma_3-\delta$ of $M_0^+$ and $M_0^-$, respectively, have absolute value at least $\bigl(\frac{|r|}{2}\bigr)^{1/d}$. Note that by \eqref{eq:star_d}
\[
|\det\, M_0|\leq (\sigma_3\cdots\sigma_d)^{2/(d-2)}\sigma_3\cdots\sigma_d\leq \biggl(\frac{|r|}{2}\biggr)^{2/d}\biggl(\frac{|r|}{2}\biggr)^{(d-2)/d}=\frac{|r|}{2}.
\]
Therefore,
\[
\abs{M_0^{\pm} - M_0} =  \delta = 2 \left(\frac{|r|}{2}\right)^{1/d}\leq 2\abs{r - \det\, M_0}^{1/d},
\]
and
\[
\abs{r - \det\, M_0^{\pm}}\leq \abs{r - \det\, M_0} + \abs{\det\, M_0 - \det\, M_0^{\pm}}.
\]
But
\begin{align*}
\abs{\det\, M_0 - \det\, M_0^{\pm}} &= \abs{\delta\sigma_1\sigma_2\sigma_4\cdots\sigma_d}\nonumber\\
&\leq 2\left(\frac{|r|}{2}\right)^{1/d}\sigma_3\sigma_3\cdots\sigma_d\nonumber\\
&\leq 2\left(\frac{|r|}{2}\right)^{1/d}\left(\frac{|r|}{2}\right)^{1/d}\biggl(\frac{|r|}{2}\biggr)^{(d-2)/d}\nonumber\\
&= |r| \leq 2\abs{r - \det\, M_0}^{1/d},
\end{align*}
where we have used the fact that
\[
\sigma_3^{d-2}\leq \sigma_3\cdots\sigma_d\leq\biggl(\frac{|r|}{2}\biggr)^{(d-2)/d}
\]
and $|\det\, M_0|\leq \abs{r}/2$.

If $|\sigma_3\pm\delta|\sigma_4\cdots\sigma_d\geq\bigl(\frac{|r|}{2}\bigr)^{(d-2)/d}$, we can continue as in Case I. If not, we repeat the argument of Step II.2 (after reordering the singular values), which can be done exactly as above. It is easy to see that after at most $(d-2)$ steps, we are in the situation of Case I. 
\end{proof}

\section{Applications}\label{sc:applications}

In the following we give precise statements and proofs of the applications mentioned in the introduction. 

\subsection{Characterization of Young measures}

We first prove our main theorem:

\begin{theorem}\label{thm:main2}
Let $1 < p < d$. Suppose that $\Omega \subset \R^d$ is open and bounded, $|\partial\Omega|=0$, and let $\nu = (\nu_x)_{x\in\Omega} \subset \Mbf^1(\R^{d \times d})$ be a $p$-Young measure. Moreover let $J_1 \colon \Omega \to [-\infty,+\infty)$, $J_2 \colon \Omega \to (-\infty,+\infty]$ be measurable and such that $J_1(x)\leq J_2(x)$ for a.e.~$x\in\Omega$. Also, assume that for $i=1,2$,
\begin{equation*}
\int_{\Omega}J_1^+(x)^{p/d} \dd x<\infty
\qquad\text{and}\qquad
\int_{\Omega}J_2^-(x)^{p/d} \dd x<\infty,
\end{equation*} 
where $J_i^{\pm}$ denotes the positive or negative part of $J_i$, respectively. Then the following statements are equivalent:
\begin{itemize}
  \item[(i)] There exists a sequence of gradients $(\nabla u_j) \subset \Lrm^p(\Omega;\R^{d \times d})$ that generates $\nu$, such that 
\[
  \qquad J_1(x) \leq \det \nabla u_j(x) \leq J_2(x) \quad\text{for all $j\in\N$ and a.e.~$x\in\Omega$. }
\]
  \item[(ii)] The conditions (I)-(IV) hold:
\begin{itemize}
\item[(I)] $\displaystyle\int_{\Omega}\int\abs{A}^p \dd \nu_x(A)<\infty$;
\item[(II)] the barycenter $[\nu](x) := \int A \dd \nu_x(A)$ is a gradient, i.e.\ there exists $\nabla u \in \Lrm^p(\Omega;\R^{d \times d})$ with $[\nu] = \nabla u$ a.e.;
\item[(III)] for every quasiconvex function $h \colon \R^{d \times d} \to \R$ with $\abs{h(A)}\leq c(1+\abs{A}^p)$, the Jensen-type inequality
\[
  \qquad\qquad h(\nabla u(x)) \leq  \int h(A) \dd \nu_x(A)  \qquad\text{holds for a.e.\ $x \in \Omega$;}
\]
\item[(IV)] $\supp{\nu_x} \subset \setb{ A \in \R^{d \times d} }{ J_1(x) \leq \det\, A \leq J_2(x) }$ for a.e.\ $x \in \Omega$.
\end{itemize}
\end{itemize}
Furthermore, in this case the sequence $(u_j)$ can be chosen such that $(\nabla u_j)$ is $p$-equiintegrable and $u_j - u\in\Wrm^{1,p}_0(\Omega,\mathbb{R}^d)$, where $u\in\Wrm^{1,p}(\Omega,\R^d)$ is the deformation underlying $\nu$ (i.e.\ the function whose gradient is the barycenter of $\nu$).
\end{theorem}

\begin{proof}
The result follows by Theorem~\ref{thm:main_abstract} and Corollary~\ref{cor:geometry}.
\end{proof}

Let us also state a refinement of the sufficiency part of the preceding theorem (and also of the main result of~\cite{KoRiWi13OPYM}):

\begin{theorem} \label{thm:kappa_equiint}
For $1 < p < d$ and a bounded open Lipschitz domain $\Omega \subset \R^d$, let $\kappa$ be a \term{singular growth modulus}, that is, a convex function $\kappa \colon (0,\infty) \to [0,\infty)$ with $\kappa(s) \to +\infty$ as $s \to 0$, which we extend by setting $\kappa(s) := +\infty$ for $s\leq 0$. Assume furthermore that we are given a Young measure $\nu = (\nu_x)_{x\in\Omega} \subset \Mbf^1(\R^{d \times d})$ satisfying~(I)--(III) from Theorem~\ref{thm:main2} as well as
\begin{equation} \label{eq:kappa_cond}
  \int \kappa( \det\, A ) \dd \nu_x(A) < \infty
  \qquad \text{for a.e.\ $x \in \Omega$.}
\end{equation}
Then, there exists a sequence of gradients $(\nabla u_j) \subset \Lrm^p(\Omega;\R^{d \times d})$ that generates $\nu$ and such that 
\begin{equation} \label{eq:kappa_equiint}
  \bigl\{ \abs{\nabla u_j}^p + \kappa( \det \nabla u_j ) \bigr\}_j
  \quad\text{is an equiintegrable family.}
\end{equation}
\end{theorem}

\begin{proof}
The condition~\eqref{eq:kappa_cond} entails in particular that $\nu_{x_0}( \setn{ A \in \R^{d \times d}}{ \det\, A > 0 })=1$ for almost every $x_0\in\Omega$. Fix such an $x_0 \in \Omega$ and denote by $A_0 := [\nu_{x_0}]  = \nabla u(x_0) \in \R^{d \times d}$ the barycenter of $\nu_{x_0}$. Let $(v_j) \subset (\Wrm^{1,p} \cap \Crm^\infty)(\Bbb^d;\R^d)$ such that $(\nabla v_j)$ is a $p$-equiintegrable generating sequence for $\nu_{x_0}$ satisfying the additional constraint $v_j(y) = [\nu_{x_0}](y) = A_0 y$ for $y \in \partial \Bbb^d$; the existence of this sequence follows from standard Young measure results, see for example Lemmas~8.3 and~8.15 in~\cite{Pedr97PMVP}.

Let $n \in \N$ and select $k = k(n) \in \N$ so large that $k(n) \geq n$, $\kappa(1/k) \geq 1$, and
\[
  \kappa(1/k) \cdot \nu_{x_0} \bigl( \setb{ A \in \R^{d \times d} }{ \det\, A \leq 1/k } \bigr)
  \leq \int_{\{\det\, A \leq 1/k\}} \kappa( \det\, A ) \dd \nu_{x_0}(A)
  \leq \frac{1}{n}.
\]
Here we have used implicitly that $\kappa$ is decreasing on an interval $(0,s_0)$ with $s_0 > 0$ since it is convex and $\kappa(s) \to +\infty$ as $s \to 0$; choose $k \geq 1/s_0$.

Using the Young measure representation of limits and discarding some elements at the beginning of the sequence $(v_j)$ if necessary, we may pick $j = j(n)$ such that with $\omega_d := \abs{\Bbb^d}$ the following two conditions hold:
\begin{align}
  &\int_{E_j} \, \absBB{ \frac{1}{k} - \det\, \nabla v_j }^{p/d}\dd y + \abs{E_j}
    \leq \frac{4 \omega_d}{\kappa(1/k) n}
    \leq \frac{C}{n}
    \qquad\text{and} 		\label{eq:YM_conv1} \\
  &\int_{\{ 1/(k+1) \leq \det\, \nabla v_j \leq 1/\ell \} } \, \kappa( \det\, \nabla v_j ) \dd y \label{eq:YM_conv2}\\
  &\qquad \leq  \omega_d \int_{\{ 1/(k+1) \leq \det\, A \leq 1/\ell \}} \kappa( \det\, A ) \dd \nu_{x_0}(A) + \frac{1}{n} \notag
\end{align}
for all $\ell = 1,\ldots,k = k(n)$, where
\[
  E_j := \setb{ y \in \Bbb^d }{ \det\, \nabla v_j(y) \leq 1/k }
\]
and $C = C(d)$ is a dimensional constant. For~\eqref{eq:YM_conv2} we used the Young measure upper semi-representation for the \emph{bounded} upper semicontinuous integrand $g(A) := \ONE_{\{ 1/(k+1) \leq \det\, A \leq 1/\ell \}} \cdot \kappa( \det\, A )$,
\[
  \limsup_{j\to\infty} \int_{\Bbb^d} g(\nabla v_j(y)) \dd y \leq \int_{\Bbb^d} \int g(A) \dd \nu_{x_0}(A) \dd x.
\]
Since the above are only finitely many conditions, they can be satisfied by discarding only finitely many leading terms in the sequence $(v_j)$. As an immediate consequence, however, the assertion~\eqref{eq:YM_conv2} in fact holds for all $\ell \in N$ since for $\ell > k+1$ it is trivially satisfied.

Next, we choose an open set $D_j$ with Lipschitz boundary and such that
\begin{align*}
  B_j &:= \setb{ y \in \Bbb^d }{ \det\, \nabla v_j(y) < 1/(k+1) } \\
  &\phantom{:}\subset D_j \subset \setb{ y \in \Bbb^d }{ \det\, \nabla v_j(y) \leq 1/k } = E_j.
\end{align*}
This is always possible: since $\nabla v_j$ is continuous, $\partial B_j$ and $\partial E_j$ can meet only in $\partial \Bbb^d$, so we can construct $D_j$ with a Lipschitz (or even smooth) boundary. Invoking Proposition~\ref{prop:convexint} for the function $v_j$ \emph{restricted to the set} $D_j$, we get a new function $w_j \in \Wrm^{1,p}(D_j;\R^d)$ with
\[
  \det\, \nabla w_j \geq 1/k \quad\text{a.e.\ in $D_j$},  \qquad
  w_j = v_j  \quad\text{on $\partial D_j$,}
\]
where as usual the boundary assertion is to be understood in the sense of trace. If $\partial D_j$ intersects $\partial \Bbb^d$, the boundary assertion is to include $w_j(y) = v_j(y) = A_0y$ for $y \in \partial D_j \cap \partial \Bbb^d$. Moreover, we have
\[
  \norm{\nabla w_j - \nabla v_j}^p_{\Lrm^p(D_j;\R^{d \times d})} \leq C_p  \int_{E_j} \, \absBB{ \frac{1}{k} - \det\, \nabla v_j(y) }^{p/d} \dd y \leq \frac{C_p}{n},
\]
where the constant $C_p = C(d,p)$ changes from expression to expression.

Now extend our new $w_j$, which at the moment is defined in $D_j$ only, to a function on all of $\Bbb^d$ by setting
\[
  w_j(y) := v_j(y)  \qquad\text{for $y \in \Bbb^d \setminus D_j$.}
\]
Since $w_j$ agrees with $v_j$ on the boundary of $D_j$, we deduce $w_j \in \Wrm^{1,p}(\Bbb^d;\R^d)$ and also
\[
  \norm{\nabla w_j - \nabla v_j}^p_p \leq \frac{C_p}{n}.
\]

We will show next the crucial fact that the family of functions $\{ \kappa( \det \nabla w_j ) \}_j$ is equiintegrable. For this, we estimate for any $\ell \in \N$ (large enough so that $\kappa$ is decreasing on $(0,1/\ell)$):
\begin{align*}
  &\int_{\{\det\, \nabla w_j \leq 1/\ell \}} \kappa( \det\, \nabla w_j(y) ) \dd y \\
    &\qquad \leq \kappa(1/k) \cdot \abs{E_j} + \int_{\{1/(k+1) \leq \det\, \nabla v_j \leq 1/\ell \}} \, \kappa( \det\, \nabla v_j(y) ) \dd y \\
    &\qquad \leq C \biggl( \frac{1}{n} + \int_{\{\det\, A \leq 1/\ell \}} \kappa( \det\, A ) \dd \nu_{x_0}(A) \biggr).
\end{align*}
Here, for the first integral we used that $\det \nabla w_j \geq 1/k$ on $D_j \subset E_j$ and~\eqref{eq:YM_conv1}; the second integral was estimated using~\eqref{eq:YM_conv2}. Now recall that $j = j(n)$ was chosen depending on $n$ (and also on $k$, but this is again chosen depending on $n$). Thus, we may take the limit superior as $n \to \infty$ or, equivalently, as $j\to\infty$, to get
\[
  \limsup_{j\to\infty} \int_{\{\det\, \nabla w_j \leq 1/\ell \}} \kappa( \det\, \nabla w_j(y) ) \dd y
    \leq C \int_{\{\det\, A \leq 1/\ell \}} \kappa( \det\, A ) \dd \nu_{x_0}(A)
\]
and this vanishes as $\ell \toup \infty$.

Hence, we conclude that $\{ \kappa( \det \nabla w_{j(n)} ) \}_n$, or, without labeling the subsequence of $j$'s, $\{ \kappa( \det \nabla w_{j} ) \}_j$, is an equiintegrable family, i.e.\ after renaming the sequence we arrive at~\eqref{eq:kappa_equiint}. More precisely, given $K > 0$ we choose $\ell \in \N$ such that $\kappa(1/\ell) < K \leq \kappa(1/(\ell+1))$, whereby $\{ \kappa(\det\, A) > K \} \subset \{ \det\, A < 1/\ell \}$ and since $\ell \toup \infty$ as $K \toup \infty$ the above assertion implies the sought equiintegrability.
\end{proof}

\subsection{Connection to the Dacorogna--Moser theory and extensions}
\label{ssc:DM}

We investigate a similar question as in~\cite{DacMos90PDEJ}; however, in subcritical Sobolev spaces, the geometric interpretation no longer holds (which is manifested in the absence of compatibility conditions on the boundary).

\begin{theorem}\label{damo}
Let $\Omega \subset \R^d$ be a bounded Lipschitz domain, $1<p<d$, $J:\Omega\to\R$ be measurable with
\begin{equation*}
\int_{\Omega}|J(x)|^{p/d} \dd x<\infty,
\end{equation*}
and let $g\in \Wrm^{1-1/p,p}(\partial\Omega;\R^d)$. Then, there exists $v\in \Wrm^{1,p}(\Omega;\R^d)$ such that
\[
\left\{
\begin{aligned}
\det\nabla v(x)&=J(x)   &&\quad\text{for a.e.\ $x \in \Omega$,}\\
v|_{\partial\Omega}&=g  &&\quad\text{in the sense of trace.}
\end{aligned}
\right.
\]
\end{theorem}
\begin{proof}
Since the trace operator is surjective from $\Wrm^{1,p}(\Omega)$ to $\Wrm^{1-1/p,p}(\partial\Omega)$, there exists $u\in \Wrm^{1,p}(\Omega)$ such that $u|_{\partial\Omega}=g$ in the trace sense. The statement then follows immediately by Corollary~\ref{cor:geometry} combined with Proposition~\ref{prop:convexint}, taking $R(x,A)=|\det\, A-J(x)|\in\Rcal^{p,d}(\Omega;\R^{d \times d})$. 
\end{proof}

\begin{corollary}
Let $\Omega \subset \R^d$ be a bounded Lipschitz domain, $1<p<d$ and $g\in \Wrm^{1-1/p,p}(\partial\Omega;\R^d)$. Then, there exists a map $v\in \Wrm^{1,p}(\Omega;\R^d)$ such that
\[
\left\{
\begin{aligned}
\det\nabla v(x)&=1      &&\quad\text{for a.e.\ $x \in \Omega$,}\\
v|_{\partial\Omega}&=g  &&\quad\text{in the sense of trace.}
\end{aligned}
\right.
\]
\end{corollary}

Of course, also the constraint $\det \nabla u(x) = J(x)$ for a given $J \colon \Omega' \to \R$, satisfying the usual assumptions, can be treated.

\subsection{Relaxation}
\label{ssc:relax}

Consider the following two functionals for a Carath\'{e}odory function $f \colon \Omega \times \R^{d\times d} \to \R$ and a function $\bar{u}\in\Wrm^{1,p}(\Omega)$:
\begin{itemize}
\item $\Fcal[u]:= \displaystyle \int_{\Omega}f(x,\nabla u(x)) \dd x$,\qquad defined over the set
\[
\mathcal{A}:=\setb{u\in \Wrm^{1,p}(\Omega,\mathbb{R}^d)}{u\vert_{\partial\Omega}=\bar{u},\,\nabla u(x)\in S_R\text{ a.e.}},
\]
where $S_R = S_{\det\geq r} :=\setb{A\in\R^{d\times d}}{\det\, A \geq r}$ or $S_R = S_{\det=r} :=\setb{A\in\R^{d\times d}}{\det\, A = r}$.
\item $\Fcal^{YM}(\nu):= \displaystyle\int_{\Omega} \int f(x,A) \dd \nu_x(A) \dd x$,\qquad defined over the set
\[
\mathcal{A}^{YM}:=\setb{\nu\,\text{$p$-GYM}}{ \text{$\supp{\nu_x}\subset S_R$ a.e.\ , $[\nu]=\nabla u$, $u\in\Wrm^{1,p}(\Omega;\R^d)$, $u\vert_{\partial\Omega}=\bar{u}$}},
\]
where we used \enquote{$p$-GYM} as an abbreviation for \enquote{gradient $p$-Young measure}.
\end{itemize}
The following relaxation theorem holds:

\begin{corollary}\label{thm:relaxation}
Suppose that $\Omega\subset\mathbb{R}^d$ is a bounded Lipschitz domain, $1<p<d$, $\bar{u}\in\Wrm^{1,p}(\Omega)$, and $f: \Omega \times \mathbb{R}^{d\times d}\rightarrow\mathbb{R}$ is a Carath\'{e}odory function satisfying
\[
c(\abs{A}^p-1)\leq f(x,A)\leq C(1+\abs{A}^p)
\]
for all $(x,A)\in\Omega\times\mathbb{R}^{d\times d}$ and constants $0<c\leq C$. Then,
\[
\inf_{\mathcal{A}}\, I=\min_{\mathcal{A}^{YM}}\, \Fcal^{YM}.
\]
In particular, whenever $(u_j)$ is an infimizing sequence of $I$ in $\mathcal{A}$, a subsequence of $(\nabla u_j)$ generates a Young measure $\nu\in\mathcal{A}^{YM}$ minimizing $\Fcal^{YM}$  over $\mathcal{A}^{YM}$. Conversely, whenever $\nu$ minimizes $\Fcal^{YM}$ in $\mathcal{A}^{YM}$, there exists an infimizing sequence $(u_j)$ of $I$ in $\mathcal{A}$ such that $(\nabla u_j)$ generates $\nu$.
\end{corollary}

\begin{proof}
Given the characterization of gradient $p$-Young measures with support in $S_R$ above, the proof is standard.
\end{proof}

Note that,  in our regime of $p<d$, the determinant is not in general weakly continuous along infimizing sequences and one cannot take
\[
\mathcal{A}^{YM}=\setb{\nu\,\text{$p$-GYM}}{ \text{$\supp{\nu_x}\subset S_R$ a.e.\ , $[\nu]=\nabla u$, where $u\in\Acal$}}
\]
as the set of admissible measures in the above relaxation theorem.

\subsection{Approximation}
\label{ssc:approx}

Next, we obtain the following interesting approximation result:

\begin{corollary}\label{cor:approximation}
Let $\Omega\subset\R^d$ be open and bounded with $|\partial\Omega|=0$. Suppose that $1 < p < d$ and $u\in\Wrm^{1,p}(\Omega,\R^d)$. For $S_R$, where either $R=r-\det\, A$ or $R = |r-\det\, A|$, there exists a sequence $(u_j)\subset\Wrm^{1,p}(\Omega,\R^d)$ bounded such that for all $j\in\N$, $u_j - u\in\Wrm^{1,p}_0$, $\nabla u_j(x)\in S_R$ for a.e.~$x\in\Omega$ and as $j\to\infty$
\[
u_j \rightharpoonup u\mbox{ in $\Wrm^{1,p}(\Omega,\R^d)$.}
\]
In particular, $\norm{u_j - u}_{p}\to 0$ as $j\to\infty$.
\end{corollary}

\begin{proof}
Let $u\in \Wrm^{1,p}(\Omega,\R^d)$ and define a gradient $p$-Young measure $(\nu_x)$ with $[\nu]=\nabla u$ by
\[
\nu_x = \left\{\begin{array}{ll} \delta_{\nabla u(x)}, &\nabla u(x)\in S_R \\ \mu_x, & \nabla u(x)\notin S_R\end{array}\right. ,
\]
where $\mu_x$ is the homogeneous gradient $p$-Young measure provided by the fact that $\R^{d\times d}$ is tightly contained in the $p$-quasiconvex hull of $S_R$ for either $S_{\det \geq r}$ (see \proofstep{Step~1} in the proof of Corollary~\ref{cor:geometry}) or $S_{\det = r}$ (see Proposition~\ref{prop:geometry}). By Theorem~\ref{thm:main2}, there exists $(u_j)\subset \Wrm^{1,p}(\Omega,\R^d)$ generating $(\nu_x)$ such that $\nabla u_j(x)\in S_R$, $u_j - u\in\Wrm^{1,p}_0(\Omega,\R^d)$ and $u_j\rightharpoonup u$ in $\Wrm^{1,p}(\Omega,\R^d)$.
\end{proof}

For simplicity, we only stated this result for the constraints $\det \geq r$ and $\det = r$ which are relevant in elasticity; nevertheless, we note that the same result holds for the more general constraint
\[
J_1(x) \leq \det\nabla u_j(x) \leq J_2(x)\qquad\text{for all $j$ and a.e.~$x$},
\]
with $J_1$, $J_2$ are as in Theorem~\ref{thm:main2}. We note that this produces arbitrary counterexamples to the weak continuity of the determinant in $\Wrm^{1,p}(\Omega,\R^d)$ for $p<d$.

\section{Lack of lower semicontinuity for a class of functionals}\label{sec:lsc}
A \term{singular growth modulus} (cf. Theorem~\ref{thm:kappa_equiint}) is a convex function $\kappa \colon (0,\infty) \to [0,\infty)$ with $\kappa(s) \to +\infty$ as $s \to 0$. We extend $\kappa$ by setting $\kappa(s) := +\infty$ for $s\leq 0$. For $p<d$, let us assume the growth condition
\begin{equation} \label{eq:kappa_growth}
  \limsup_{s \to +\infty}\, \frac{\kappa(s)}{s^{p/d}} < \infty.
\end{equation}
In what follows, $f \colon \Omega \times \R^{d \times d} \to [0,\infty]$ will be a Carath\'{e}odory integrand satisfying the \term{elastic coercivity/growth estimates}
\begin{equation} \label{eq:f_elasticity_est}
  \frac{1}{M} \bigl( \abs{A}^p + \kappa( \det\, A ) \bigr) \leq f(x,A) \leq M \bigl( 1 + \abs{A}^p + \kappa( \det\, A ) \bigr)
\end{equation}
for a constant $M > 0$.

In this section, we want to show that under these assumptions, the functional
\begin{equation} \label{eq:F_def}
  \Fcal[u] := \int_\Omega f(x,\nabla u(x)) \dd x,	\qquad
  \text{where $u \in \Wrm^{1,p}(\Omega;\R^d)$ with $\det\, \nabla u > 0$ a.e.,}
\end{equation}
is \emph{not} $\Wrm^{1,p}$-weakly lower semicontinuous along sequences $u_j \toweak u$ in $\Wrm^{1,p}(\Omega;\R^d)$ satisfying the additional constraint
\[
  \det\, \nabla u > 0  \quad\text{a.e.}
\]
We show this in two steps: First we show that this form of lower semicontinuity implies a certain quasiconvexity condition on $f$; secondly, we prove that no such integrands exist under the growth conditions~\eqref{eq:f_elasticity_est}. 

More precisely, let $h \colon \R^{d \times d} \to (-\infty,+\infty]$ be a Borel function that is locally bounded on (i.e.\ bounded on any compact subset of) the set $\setn{ A \in \R^{d \times d} }{ \det\, A > 0 }$. We call $h$ \term{$\Wrm^{1,p}$-orientation-preserving quasiconvex} if
\[
  h(A_0) \leq \dashint_{\Bbb^d} h(\nabla v(x)) \dd x
\]
for all $A_0 \in \R^{d \times d}$ with $\det\, A_0 > 0$ and all $v \in \Wrm^{1,p}(\Bbb^d;\R^d)$ with $v(x) = A_0x$ on $\partial \Bbb^d$ (in the sense of trace) and $\det\, \nabla v > 0$ almost everywhere (recall that $\Bbb^d$ denotes the unit ball in $\R^d$).

We note that under the additional $p$-growth condition $\abs{h(A)} \leq M(1+\abs{A}^p)$ the notion of $\Wrm^{1,p}$-orientation-preserving quasiconvexity is weaker than the usual quasiconvexity~\cite{Morr52QSMI,Daco08DMCV}, since it is clearly weaker than $\Wrm^{1,p}$-quasiconvexity.

\begin{remark}
We remark that, starting from the prototypical example of the determinant, there is a sizeable literature on the weak lower semicontinuity of polyconvex and quasiconvex functionals below the critical exponent $p=d$. As this lies outside the scope of the present work the reader is referred to \cite{FM97,Marcellini86,Maly93,AcDalM94,DalMSb95} and references therein. 
\end{remark}

Returning to our result, we then have:
\begin{proposition} \label{prop:converse}
For $1 < p < d$ and a bounded open Lipschitz domain $\Omega \subset \R^d$, let $\kappa$ be a singular growth modulus with~\eqref{eq:kappa_growth} and assume that $f \colon \Omega \times \R^{d \times d} \to [0,\infty)$ is a Carath\'{e}odory integrand satisfying the elastic coercivity/growth estimates~\eqref{eq:f_elasticity_est}. Also, let the functional $\Fcal$ be defined as in~\eqref{eq:F_def}. If $\Fcal$ is $\Wrm^{1,p}$-weakly lower semicontinuous along sequences $u_j \toweak u$ in $\Wrm^{1,p}(\Omega;\R^3)$ satisfying the additional assumption $\det\, \nabla u > 0$ a.e., then
\[
  \text{$f(x,\frarg)$ is orientation-preserving quasiconvex for almost every $x \in \Omega$.}
\]
\end{proposition}

\begin{proof}
We assume that $f(x,A) = h(A)$ does not depend on $x$ (otherwise, one needs to use an additional localization argument).

Let $A_0 \in \R^{d \times d}$ with $\det\, A_0 > 0$ and let $v \in \Wrm^{1,p}(\Bbb^d;\R^d)$ with $v(x) = A_0x$ on $\partial \Bbb^d$ (in the sense of trace) and $\det\, \nabla v > 0$ a.e. By virtue of the Vitali Covering Theorem, find a covering of $\Lcal^d$-almost all of $\Bbb^d$ by balls $B(x_k,r_k) \subset \Bbb^d$ such that $r_k \leq 1/j$, $k \in \N$, and define
\[
  w_j(x) := \sum_k r_k \, \ONE_{B(x_k,r_k)}(x) \, v \Bigl( \frac{x-x_k}{r_k} \Bigr) + A_0 x_k,
\]
hence $w_j(x) = A_0 x$ for $x \in \partial \Bbb^d$ (in the sense of trace) and
\[
  \nabla w_j(x) = \sum_k \ONE_{B(x_k,r_k)}(x) \nabla v \Bigl( \frac{x-x_k}{r_k} \Bigr).
\]
Then,
\begin{align*}
  \int_{\Bbb^d} h(\nabla w_j(x)) \dd x &= \sum_k \int_{B(x_k,r_k)} h \Bigl( \nabla v \Bigl( \frac{x-x_k}{r_k} \Bigr) \Bigr) \dd x
    = \sum_k r_k^d \int_{\Bbb^d} h(\nabla v(y)) \dd y \\
  &= \int_{\Bbb^d} h(\nabla v(y)) \dd y
\end{align*}
Also, $w_j$ converges weakly to the linear function $x \mapsto A_0x$ in $\Wrm^{1,p}(\Bbb^d;\R^d)$. Thus, since $\det\, \nabla u(x) = \det\, A_0 > 0$, the assumed lower semicontinuity implies
\[
  h(A_0) \leq \liminf_{j\to\infty} \dashint_{\Bbb^d} h(\nabla w_j(x)) \dd x = \dashint_{\Bbb^d} h(\nabla v(y)) \dd y,
\]
and $h$ is $\Wrm^{1,p}$-orientation-preserving quasiconvex.
\end{proof}

Next, we prove:

\begin{proposition}\label{prop:nonexistence}
Suppose that $h$ satisfies the growth conditions~\eqref{eq:f_elasticity_est} for some $p\in (1,d)$. Then $h$ is not $\Wrm^{1,p}$-orientation-preserving quasiconvex.
\end{proposition}
\begin{proof}
For $\epsilon>0$, define $A^{\epsilon}=\epsilon I$, where $I$ denotes the $d\times d$ identity matrix. By Propositions~\eqref{prop:convexint} and~\ref{prop:geometry} with $r=1$, there exists $v^{\epsilon}\in \Wrm^{1,p}(\Bbb^d)$ such that $\det\,\nabla v^{\epsilon}=1$ almost everywhere and $v^{\epsilon}-A^{\epsilon}x\in\Wrm_0^{1,p}(\Bbb^d)$. Moreover, by Proposition~\eqref{prop:convexint},
\begin{equation*}
\norm{\nabla v^{\epsilon}-A^{\epsilon}}_p^p\leq C\int_{\Bbb^d}|\det\, A^{\epsilon}-1|^{p/d}\dd x,
\end{equation*}
whence it follows (observing $\det\, A^{\epsilon}=\epsilon^d$) that $\norm{\nabla v^{\epsilon}}_p<C$ for a constant independent of $\epsilon$ (at least when $\epsilon$ is small). By~\eqref{eq:f_elasticity_est}, on one hand
\begin{equation*}
\lim_{\epsilon\searrow0}h(A^{\epsilon})=+\infty,
\end{equation*}
but on the other hand
\begin{equation*}
\begin{aligned}
\dashint_{\Bbb^d} h(\nabla v^{\epsilon}(x)) \dd x&\leq M\dashint_{\Bbb^d}(1+|\nabla v^{\epsilon}(x)|+\kappa(1))\\
&\leq M(1+\kappa(1)+\norm{\nabla v^{\epsilon}}_p^p)\leq C.
\end{aligned}
\end{equation*}
Since $\epsilon>0$ was arbitrary, it follows that $h$ cannot be $\Wrm^{1,p}$-orientation-preserving quasiconvex.
\end{proof}

The combination of Propositions~\ref{prop:converse} and~\ref{prop:nonexistence} finally yields:
\begin{theorem}\label{thm:wlsc}
For $1 < p < d$ and a bounded open Lipschitz domain $\Omega \subset \R^d$, let $\kappa$ be a singular growth modulus with~\eqref{eq:kappa_growth} and assume that $f \colon \Omega \times \R^{d \times d} \to [0,\infty)$ is a Carath\'{e}odory integrand satisfying the elastic coercivity/growth estimates~\eqref{eq:f_elasticity_est}. Also, let the functional $\Fcal$ be defined as in~\eqref{eq:F_def}. Then, $\Fcal$ is not $\Wrm^{1,p}$-weakly lower semicontinuous along sequences $u_j \toweak u$ in $\Wrm^{1,p}(\Omega;\R^3)$ satisfying the additional assumption $\det\, \nabla u > 0$ a.e. 
\end{theorem}
\begin{remark}
\begin{itemize}
\item[a)] Since every $\Wrm^{1,p}$-quasiconvex function, cf.~\cite{BalMur84WQVP}, is clearly $\Wrm^{1,p}$-orientation-preserving quasiconvex, it follows from the theorem that there exist no $\Wrm^{1,p}$-quasiconvex functions with the growth conditions~\eqref{eq:f_elasticity_est}.
\item[b)] It is apparent from the proof of Proposition~\ref{prop:nonexistence} that the theorem still holds if the upper bound in~\eqref{eq:f_elasticity_est} is weakened to $f(x,A)\leq M(1+|A|^p)$ for all matrices such that $\det\, A=r$ for \emph{some} $r>0$.

\end{itemize}
\end{remark}


\begin{thebibliography}{{Mor}52}

\bibitem[ADM94]{AcDalM94}
E.~Acerbi and G.~Dal~Maso, \emph{New lower semicontinuity results for
  polyconvex integrals}, Calculus of Variations and Partial Differential
  Equations \textbf{2} (1994), no.~3, 329--371.

\bibitem[AFS08]{AsFaSz08CILP}
K.~Astala, D.~Faraco, and L.~{Sz\'{e}kelyhidi, Jr.}, \emph{Convex integration
  and the $l^p$ theory of elliptic equations}, Ann. Sc. Norm. Super. Pisa Cl.
  Sci. \textbf{7} (2008), 1--50.

\bibitem[AM08]{AnHMan09RTNE}
O.~{Anza Hafsa} and J.-P. Mandallena, \emph{Relaxation theorems in nonlinear
  elasticity}, Ann. Inst. H. Poincar\'e Anal. Non Lin\'eaire \textbf{25}
  (2008), 135--148.

\bibitem[Bal82]{Ball82DESC}
J.~M. Ball, \emph{Discontinuous equilibrium solutions and cavitation in
  nonlinear elasticity}, Philos. Trans. Roy. Soc. London Ser. A \textbf{306}
  (1982), 557--611.

\bibitem[Bal02]{Ball02SOPE}
\bysame, \emph{Some open problems in elasticity}, Geometry, mechanics, and
  dynamics, Springer, 2002, pp.~3--59.

\bibitem[Bal77]{Ball77CCET}
J.~M. Ball, \emph{Convexity conditions and existence theorems in nonlinear
  elasticity}, Arch. Ration. Mech. Anal. \textbf{63} (1976/77), 337--403.

\bibitem[BM84]{BalMur84WQVP}
J.~M. Ball and F.~Murat, \emph{{$W^{1,p}$}-quasiconvexity and variational
  problems for multiple integrals}, J. Funct. Anal. \textbf{58} (1984),
  225--253, Erratum: Vol. 66 (1986), 439.

\bibitem[CD14]{ContiDolz14}
S.~Conti and G.~Dolzmann, \emph{On the theory of relaxation in nonlinear
  elasticity with constraints on the determinant}, arXiv preprint
  arXiv:1403.5779 (2014).

\bibitem[CDK09]{CuDaKn09OENS}
G.~Cupini, B.~Dacorogna, and O.~Kneuss, \emph{On the equation {$\det\nabla
  u=f$} with no sign hypothesis}, Calc. Var. Partial Differential Equations
  \textbf{36} (2009), 251--283.

\bibitem[Dac08]{Daco08DMCV}
B.~Dacorogna, \emph{{Direct Methods in the Calculus of Variations}}, 2nd ed.,
  Applied Mathematical Sciences, vol.~78, Springer, 2008.

\bibitem[DM90]{DacMos90PDEJ}
B.~Dacorogna and J.~Moser, \emph{On a partial differential equation involving
  the {J}acobian determinant}, Ann. Inst. H. Poincar\'e Anal. Non Lin\'eaire
  \textbf{7} (1990), 1--26.

\bibitem[DM97]{DacMar97GETH}
B.~Dacorogna and P.~Marcellini, \emph{General existence theorems for
  hamilton-jacobi equations in the scalar and vectorial cases}, Acta Math.
  \textbf{178} (1997), 1--37.

\bibitem[DMS95]{DalMSb95}
G.~Dal~Maso and C.~Sbordone, \emph{Weak lower semicontinuity of polyconvex
  integrals: a borderline case}, Mathematische Zeitschrift \textbf{218} (1995),
  no.~1, 603--609.

\bibitem[DP12]{DePhil12WJ}
G.~De~Philippis, \emph{Weak notions of {J}acobian determinant and relaxation},
  ESAIM Control Optim. Calc. Var. \textbf{18} (2012), 181--207.

\bibitem[DS12]{DeLSze12HEFD}
C.~{De Lellis} and L.~{Sz\'{e}kelyhidi, Jr.}, \emph{The {$h$}-principle and the
  equations of fluid dynamics}, Bull. Amer. Math. Soc. (N.S.) \textbf{49}
  (2012), 347--375.

\bibitem[EM02]{EliMis02IH}
Y.~Eliashberg and N.~Mishachev, \emph{Introduction to the {$h$}-principle},
  Graduate Studies in Mathematics, vol.~48, American Mathematical Society,
  2002.

\bibitem[FLM05]{FoLeMa05WCLS}
I.~Fonseca, G.~Leoni, and J.~{Mal\'{y}}, \emph{Weak continuity and lower
  semicontinuity results for determinants}, Arch. Ration. Mech. Anal.
  \textbf{178} (2005), 411--448.

\bibitem[FM97]{FM97}
I.~Fonseca and J.~Mal{\`y}, \emph{Relaxation of multiple integrals below the
  growth exponent}, Ann. Inst. H. Poincar\'e Anal. Non Lin\'eaire \textbf{14}
  (1997), 309--338.

\bibitem[FM99]{FonMul99AQLS}
I.~Fonseca and S.~M{\"u}ller, \emph{{$\mathcal{A}$}-quasiconvexity, lower
  semicontinuity, and {Y}oung measures}, SIAM J. Math. Anal. \textbf{30}
  (1999), no.~6, 1355--1390.

\bibitem[FMP98]{FoMuPe98ACOE}
I.~Fonseca, S.~M{\"u}ller, and P.~Pedregal, \emph{Analysis of concentration and
  oscillation effects generated by gradients}, SIAM J. Math. Anal. \textbf{29}
  (1998), 736--756.

\bibitem[GMS98]{GMSvol1}
M.~Giaquinta, G.~Modica, and J.~Sou{\v{c}}ek, \emph{Cartesian {C}urrents in the
  {C}alculus of {V}ariations, {I}, {C}artesian {C}urrents}, Ergebnisse der
  Mathematik und ihrer Grenzgebiete, vol.~37, Springer, 1998.

\bibitem[Gro86]{Grom86PDR}
M.~Gromov, \emph{Partial differential relations}, Ergebnisse der Mathematik und
  ihrer Grenzgebiete, vol.~9, Springer, 1986.

\bibitem[Hen11]{Hen11SobHomJ0}
S.~Hencl, \emph{Sobolev homeomorphism with zero {J}acobian almost everywhere},
  J. Math. Pures Appl. \textbf{95} (2011), 444--458.

\bibitem[HMC10]{HeMo-Co10CAV&FRAC}
D.~Henao and C.~Mora-Corral, \emph{Invertibility and weak continuity of the
  determinant for the modelling of cavitation and fracture in nonlinear
  elasticity}, Arch. Ration. Mech. Anal. \textbf{197} (2010), no.~2, 619--655.

\bibitem[Kir03]{Kirc03RGM}
B.~Kirchheim, \emph{{Rigidity and Geometry of Microstructures}}, Lecture
  notes~16, {Max-Planck-Institut f\"{u}r Mathematik in den Naturwissenschaften,
  Leipzig}, 2003.

\bibitem[Kne12]{Kneu12OEPB}
O.~Kneuss, \emph{On the equation {$\det\nabla\phi=f$} prescribing {$\phi=0$} on
  the boundary}, Differential Integral Equations \textbf{25} (2012),
  1037--1052.

\bibitem[KP91]{KinPed91CYMG}
D.~Kinderlehrer and P.~Pedregal, \emph{Characterizations of {Y}oung measures
  generated by gradients}, Arch. Ration. Mech. Anal. \textbf{115} (1991),
  329--365.

\bibitem[KP94]{KinPed94GYMG}
\bysame, \emph{Gradient {Young} measures generated by sequences in {Sobolev}
  spaces}, J. Geom. Anal. \textbf{4} (1994), 59--90.

\bibitem[KR10]{KriRin10CGGY}
J.~Kristensen and F.~Rindler, \emph{Characterization of generalized gradient
  {Young} measures generated by sequences in {W$\sp{1,1}$} and {BV}}, Arch.
  Ration. Mech. Anal. \textbf{197} (2010), 539--598, Erratum: Vol. 203 (2012),
  693-700.

\bibitem[KRW13]{KoRiWi13OPYM}
K.~Koumatos, F.~Rindler, and E.~Wiedemann, \emph{Orientation-preserving {Young}
  measures}, submitted (August 2013), arXiv:1307.1007, 2013.

\bibitem[Mal93]{Maly93}
J.~Mal{\`y}, \emph{Weak lower semicontinuity of polyconvex integrals}, Proc.
  Roy. Soc. Edinburgh Sect. A \textbf{123} (1993), 681--691.

\bibitem[Mar86]{Marcellini86}
P.~Marcellini, \emph{On the definition and the lower semicontinuity of certain
  quasiconvex integrals}, Ann. Inst. H. Poincar\'e Anal. Non Lin\'eaire
  \textbf{3} (1986), 391--409.

\bibitem[{Mor}52]{Morr52QSMI}
C.~B. {Morrey, Jr.}, \emph{Quasiconvexity and the semicontinuity of multiple
  integrals}, Pacific J. Math. \textbf{2} (1952), 25--53.

\bibitem[Mos65]{Mose65OVEM}
J.~Moser, \emph{On the volume elements on a manifold}, Trans. Amer. Math. Soc.
  \textbf{120} (1965), 286--294. \MR{0182927 (32 \#409)}

\bibitem[MS95]{MuSp95EX}
S.~M{\"u}ller and S.~J. Spector, \emph{An existence theory for nonlinear
  elasticity that allows for cavitation}, Arch. Ration. Mech. Anal.
  \textbf{131} (1995), no.~1, 1--66.

\bibitem[M{\v{S}}03]{MulSve03CILM}
S.~M{\"u}ller and V.~{\v{S}}ver{\'a}k, \emph{Convex integration for {L}ipschitz
  mappings and counterexamples to regularity}, Ann. of Math. \textbf{157}
  (2003), 715--742.

\bibitem[M{\"u}l90]{Mull90DETD}
S.~M{\"u}ller, \emph{{${\rm Det}={\rm det}$}. a remark on the distributional
  determinant}, C. R. Acad. Sci. Paris S\'{e}r. I Math. \textbf{311} (1990),
  13--17.

\bibitem[M{\"u}l93]{Mull93SSDD}
\bysame, \emph{On the singular support of the distributional determinant}, Ann.
  Inst. H. Poincar\'e Anal. Non Lin\'eaire \textbf{10} (1993), 657--696.

\bibitem[M{\"u}l99]{Mull99VMMP}
S.~M{\"u}ller, \emph{Variational models for microstructure and phase
  transitions}, Calculus of variations and geometric evolution problems
  (Cetraro, 1996), Lecture Notes in Mathematics, vol. 1713, Springer, 1999,
  pp.~85--210.

\bibitem[Nas54]{Nash54C1IE}
J.~Nash, \emph{{$C\sp 1$} isometric imbeddings}, Ann. of Math. \textbf{60}
  (1954), 383--396.

\bibitem[Ped97]{Pedr97PMVP}
P.~Pedregal, \emph{{Parametrized Measures and Variational Principles}},
  Progress in Nonlinear Differential Equations and their Applications, vol.~30,
  Birkh\"auser, 1997.

\bibitem[Rin14]{Rind14LPCY}
F.~Rindler, \emph{A local proof for the characterization of {Young measures}
  generated by sequences in {BV}}, J. Funct. Anal. \textbf{266} (2014),
  6335--6371.

\bibitem[SS00]{SivSp00EX}
J.~Sivaloganathan and S.~J. Spector, \emph{On the existence of minimizers with
  prescribed singular points in nonlinear elasticity}, J. Elast. and Phys. Sc.
  Sol. \textbf{59} (2000), no.~1-3, 83--113.

\bibitem[{\v{S}}ve88]{Sver88RPDF}
V.~{\v{S}}ver{\'a}k, \emph{Regularity properties of deformations with finite
  energy}, Arch. Rational Mech. Anal. \textbf{100} (1988), 105--127.

\bibitem[SW12]{SzeWie12YMGI}
L.~{Sz\'{e}kelyhidi Jr.} and E.~Wiedemann, \emph{Young measures generated by
  ideal incompressible fluid flows}, Arch. Ration. Mech. Anal. \textbf{206}
  (2012), 333--366.

\bibitem[Y\'94]{Ye94PJSP}
D.~Y\'{e}, \emph{Prescribing the {Jacobian} determinant in {Sobolev} spaces},
  Ann. Inst. H. Poincar\'e Anal. Non Lin\'eaire \textbf{11} (1994), 275--296.

\bibitem[Yan96]{Yan96RemarksStability}
B.~Yan, \emph{Remarks on {$W^{1,p}$}-stability of the conformal set in higher
  dimensions}, Ann. Inst. H. Poincar\'e Anal. Non Lin\'eaire \textbf{13}
  (1996), 691--705.

\bibitem[Yan01]{Yan01LinearBVP}
\bysame, \emph{A linear boundary value problem for weakly quasiregular mappings
  in space}, Calc. Var. Partial Differential Equations \textbf{13} (2001),
  295--310.

\bibitem[Yan03]{Yan03Baire}
\bysame, \emph{A {B}aire's category method for the {D}irichlet problem of
  quasiregular mappings}, Trans. Amer. Math. Soc. \textbf{355} (2003), no.~12,
  4755--4765.

\end{thebibliography}

\providecommand{\bysame}{\leavevmode\hbox to3em{\hrulefill}\thinspace}
\providecommand{\MR}{\relax\ifhmode\unskip\space\fi MR }
\providecommand{\MRhref}[2]{%
  \href{http://www.ams.org/mathscinet-getitem?mr=#1}{#2}
}
\providecommand{\href}[2]{#2}

\end{document}